\documentclass{article}
\usepackage{graphicx} 

\usepackage{amsmath,amssymb,fullpage,xcolor}
\usepackage{amsthm,enumitem}
\definecolor{darkgreen}{RGB}{51,117,56}
\definecolor{burgundy}{RGB}{46,37,113}
\definecolor{babyblue}{RGB}{30,144,255}
\definecolor{beige}{RGB}{220,205,125}
\definecolor{burgundy}{RGB}{126,041,084}
\definecolor{pinkcheeks}{RGB}{194,106,119}
\definecolor{realpurple}{RGB}{159,074,150} 
\definecolor{babyteal}{RGB}{093,168,153}
\usepackage{tikz,verbatim}
\usetikzlibrary{decorations.pathreplacing}
\usetikzlibrary{decorations.markings}
\usetikzlibrary{arrows}

\usepackage{ytableau, ifthen}
\usepackage{hyperref}
\usepackage{stmaryrd}
\usepackage{subcaption}

\ytableausetup{centertableaux, smalltableaux}

\newtheorem{thm}{Theorem}[section]
\newtheorem{lem}[thm]{Lemma}
\newtheorem{cor}[thm]{Corollary}
\newtheorem{prop}[thm]{Proposition}

\newtheorem*{thmA}{Theorem \ref{thm:A}}
\newtheorem*{thmB}{Theorem \ref{thm:B}}

\newtheorem*{thmICAn}{Theorem \ref{thm:ICAn}}

\newtheorem*{cor3xn}{Corollary \ref{cor:3xncor}}

\theoremstyle{definition}
\newtheorem{definition}[thm]{Definition}
\newtheorem{example}[thm]{Example}

\newcommand{\IC}{\mathcal{IC}}

\newcommand{\f}{\nabla}
\newcommand{\oi}{\Delta}

\newcommand\Lmn{\mathcal{L}_{m,n}}
\newcommand\Lmnr{\mathcal{L}_{m,n;r}}
\newcommand\LLmn{\mathcal{L}^{2}_{m,n}}
\newcommand\LLmnr{\mathcal{L}^{2}_{m,n;r}}
\newcommand\MMl{\mathcal{M}^{2}_\ell}
\newcommand\MMmn{\mathcal{M}^{2}_{m,n}}
\newcommand\MMn{\mathcal{M}^{2}_{2n}}
\newcommand\MM{\mathcal{M}^{2}}
\newcommand\tMM{\widetilde{\mathcal{M}}^{2}}
\newcommand\tMMl{\widetilde{\mathcal{M}}^{2}_\ell}
\newcommand\tMMmn{\widetilde{\mathcal{M}}^{2}_{m,n}}
\renewcommand\SS{\mathcal{S}^{2}}
\newcommand\SSn{\mathcal{S}^{2}_n}
\newcommand\tSS{\widetilde{\SS}}
\newcommand\tSSn{\widetilde{\SSn}}
\newcommand\card[1]{\left|#1\right|}
\newcommand{\bA}{\mathbf A}
\newcommand{\fB}{\mathfrak B}
\newcommand{\bB}{\mathbf B}
\newcommand\Dn{\mathcal{D}_{n}}
\newcommand\DDn{\mathcal{D}^{2}_{n}}
\newcommand\Wo{\mathcal{W}^0}
\newcommand\W{\mathcal{W}}
\newcommand\tW{\widetilde{\mathcal{W}}}
\newcommand\tWo{\widetilde{\mathcal{W}}^0}
\newcommand\tWu{\widetilde{\mathcal{W}}}
\newcommand{\e}{\textnormal{\texttt{e}}}
\newcommand{\w}{\textnormal{\texttt{w}}}
\newcommand{\nw}{\textnormal{\texttt{nw}}}
\newcommand{\se}{\textnormal{\texttt{se}}}
\newcommand{\uu}{\textnormal{\texttt{u}}}
\newcommand{\dd}{\textnormal{\texttt{d}}}
\newcommand{\hh}{\textnormal{\texttt{h}}}

\title{Enumeration of interval-closed sets via Motzkin paths and quarter-plane walks}
\author{Sergi Elizalde$^a$ \and Nadia Lafreni\`ere$^b$ \and Joel Brewster Lewis$^c$ \and Erin McNicholas$^d$ \and Jessica Striker$^e$ \and Amanda Welch$^f$}

\date{\small $^a$ Dartmouth College, Department of Mathematics, 6188 Kemeny Hall, Hanover, NH 03755, USA. sergi.elizalde@dartmouth.edu\\
$^b$ Concordia University, Department of Mathematics and Statistics, 1455 De Maisonneuve Blvd.\ W., Montreal, Quebec H3G 1M8, Canada. nadia.lafreniere@concordia.ca\\
$^c$ The George Washington University, Department of Mathematics, 801 22nd St.\ NW, Washington, DC, USA. jblewis@gwu.edu\\
$^d$ Willamette University, Department of Mathematics, 900 State St, Salem, Oregon 97301, USA. emcnicho@willamette.edu\\
$^e$ North Dakota State University, Department of Mathematics, 1340 Administration Ave, Fargo, ND 58105, USA. jessica.striker@ndsu.edu\\
$^f$ Eastern Illinois University, Department of Mathematics and Computer Science, 600 Lincoln Avenue, Charleston IL, 61920, USA. arwelch@eiu.edu\\
}

\begin{document}

\maketitle

\begin{abstract}
    We find a generating function for interval-closed sets of the product of two chains poset by constructing a bijection to certain bicolored Motzkin paths. We also find a functional equation for the generating function of interval-closed sets of truncated rectangle posets, including the type $A$ root poset, by constructing a bijection to certain quarter-plane walks.
\end{abstract}

\section{Introduction}
Interval-closed sets of partially ordered sets, or posets, are an interesting generalization of both order ideals (downward-closed subsets) and order filters (upward-closed subsets). Also called convex subsets, the interval-closed sets of a poset $P$ are defined to be the subsets $I\subseteq P$ such that if $x,y\in I$ and there is an element $z$ with $x<z<y$, then $z\in I$. In other words, $I$ contains all elements of $P$ between any two elements of $I$. Interval-closed sets are important in operations research and arise in applications such as project scheduling and assembly line balance
\cite{Convex2015}.

Although order ideals of posets have been well studied from enumerative, bijective, and dynamical perspectives, interval-closed sets have not received as much attention. 
A recent paper \cite{ELMSW} initiated the study of interval-closed sets of various families of posets from enumerative and dynamical perspectives. 
In this paper, we continue to study the enumeration of interval-closed sets of specific families of posets, finding useful bijections along the way, while in the companion paper \cite{LLMSW}, we extend the study of interval-closed set rowmotion dynamics.

The main results of the present paper include a generating function for interval-closed sets of the product of two chains poset $[m]\times[n]$, from which we extract explicit formulas for small values of $m$, and functional equations for the generating functions of interval-closed sets of truncated rectangle posets, a family that includes the type $A$ root posets. In both cases, we define bijections from interval-closed sets to various kinds of lattice paths, namely, certain bicolored Motzkin paths and quarter-plane walks.

Our first main result, stated as Theorem~\ref{thm:Motzkin_bijection}, is a bijection between the set of interval-closed sets of $[m]\times[n]$ and the set of bicolored Motzkin paths with certain restrictions; specifically, the number of up steps and horizontal steps of the first color is $m$, the number of down steps and horizontal steps of the second color is $n$, and no horizontal step of the second color on the $x$-axis is followed by a horizontal step of the first color.
We use this bijection to find the following generating function.
\begin{thmA} The generating function of interval-closed sets of  $[m]\times[n]$ is given by
    $$\sum_{m,n\ge0} \card{\IC([m]\times[n])}\, x^m y^n=\frac{2}{1-x-y+2xy+\sqrt{(1-x-y)^2-4xy}}.$$
\end{thmA}

One may use this generating function to extract counting formulas for fixed values of $m$, such as the following result.

\begin{cor3xn}
The cardinality of $\IC([3]\times[n])$ is
$$\frac{n^{6}+9 n^{5}+61 n^{4}+159 n^{3}+370 n^{2}+264 n +144}{144}.$$
\end{cor3xn}

Let $\fB_n$ denote the type $B_n$ minuscule poset (illustrated in Figure~\ref{fig:B_minuscule}), whose interval-closed sets are in bijection with vertically symmetric interval-closed sets of $[n]\times[n]$.
\begin{thmB} The generating function of interval-closed sets of $\fB_n$ is given by
    $$\sum_{n\ge0} \card{\IC(\fB_n)}\, x^n=\frac{4-10x+8x^2}{2-11x+14x^2-8x^3-(2-3x)\sqrt{1-4x}}.$$
\end{thmB}

Let $\bA_n$ denote the type $A_n$ positive root poset (illustrated in Figure~\ref{fig:A14}). In Theorem~\ref{thm:walks_bijection}, we construct a bijection between the set of interval-closed sets of $\bA_{n-1}$ and the set of lattice walks in the first quadrant that start and end at the origin and consist of $2n$ steps from the set $\{ (1,0),(-1,0),(1,-1),(-1,1)\}$, where no $(-1,0)$ step on the $x$-axis is immediately followed by a $(1,0)$ step.
We use this bijection to derive the following functional equation for the generating function.
\begin{thmICAn}
    The generating function of interval-closed sets of $\bA_{n-1}$ can be expressed as 
    $$\sum_{n\ge0} \card{\IC(\bA_{n-1})}z^{2n}=F(0,0,z),$$
    where $F(x,y):=F(x,y,z)$ satisfies the functional equation
\begin{equation*}
F(x,y)= 1+z\left(x+\frac{1}{x}+\frac{x}{y}+\frac{y}{x}\right)F(x,y) - z \left(\frac{1}{x}+\frac{y}{x}\right)F(0,y)  - z\, \frac{x}{y} F(x,0)   - z^2\, \left(F(x,0)-F(0,0)\right).
\end{equation*}
\end{thmICAn}
We derive in Theorems~\ref{thm:walks_bijection_truncated} and~\ref{thm:ICP} generalizations of these theorems to the poset obtained by truncating the bottom $d$ ranks from $[m] \times [n]$. (Note that $\bA_{n-1}$ may be obtained by truncating the bottom $n$ ranks from $[n]\times[n]$.)
We also find a similar functional equation in Theorem~\ref{thm:BrootGF} for symmetric ICS of $\bA_{n-1}$ and use this to extract the enumeration of ICS of the type $B$ positive root poset (illustrated in Figure~\ref{ex_typeB}). 

The paper is organized as follows. Section~\ref{sec:def} gives necessary poset-theoretic definitions and states relevant enumerative theorems from \cite{ELMSW}. Section~\ref{sec:rectangle} studies interval-closed sets of $[m]\times[n]$ and their corresponding bicolored Motzkin paths, proving the bijection of Theorem~\ref{thm:Motzkin_bijection}, and the generating functions of Theorems \ref{thm:A} and \ref{thm:B}. It also proves Theorem \ref{thm:Motzkin_stats_bijection}, which translates statistics of interest on each side of the bijection.
Section~\ref{sec:TypeAroot} studies  interval-closed sets of {the type $A$ root posets} and truncated rectangle posets, proving Theorems~\ref{thm:walks_bijection} and \ref{thm:ICAn} on the poset $\bA_{n-1}$, Theorem \ref{thm:BrootGF} on symmetric ICS of $\bA_{n-1}$, and Theorems \ref{thm:walks_bijection_truncated} and \ref{thm:ICP} on truncated rectangle posets. Section~\ref{sec:TypeAroot} also contains Theorem~\ref{statistics_walks}, which again translates statistics across the relevant bijection.  We end in Section~\ref{sec:future} with some ideas for future work.

\section{Definitions and background}
\label{sec:def}
Let $P$ be a partially ordered set (poset). All posets in this paper are finite. Below we introduce the poset-theoretic definitions that are most relevant to this paper, and refer to \cite[Ch.\ 3]{Stanley2011} for a more thorough discussion.
\begin{definition} \label{def:ics} Let $I\subseteq P$. We say that $I$ is an \emph{interval-closed set (ICS)} of $P$ if for all $x, y \in I$ and $z\in P$ such that $x < z < y$, we have $z \in I$. Let $\IC(P)$ denote the set of all interval-closed sets of $P$.
\end{definition}

\begin{definition}\label{def:oi_of}
 A subset $J\subseteq P$ is an \emph{order ideal} if whenever $b\in J$ and $a\leq b$, we have $a\in J$. A subset $K$ is an \emph{order filter} if whenever $a\in K$ and $a\leq b$, we have $b\in K$.
 Given $S\subseteq P$, let $\oi(S)$ denote the smallest order ideal containing  $S$, and let $\f(S)$ denote the smallest order filter containing $S$.
\end{definition}

\begin{definition}\label{def:chain}
The $n$-element \textit{chain poset} has elements $1<2<\cdots<n$ and is denoted by $[n]$. In this paper, we study the poset constructed as the \emph{Cartesian product} of two chains. Its elements are $[m]\times [n]=\{(i,j) \ | \ 1\leq i\leq m, 1\leq j\leq n\}$, and the partial order is given by $(a,b)\leq (c,d)$ if and only if $a\leq c$ and $b\leq d$.
\end{definition}

Our convention is to draw the Hasse diagram of $[m]\times[n]$ as a tilted rectangle with poset element $(1,1)$ at the bottom, incrementing the first coordinate in the northeast direction and the second coordinate in the northwest direction, as in Figure \ref{fig:ex_ICS}. 

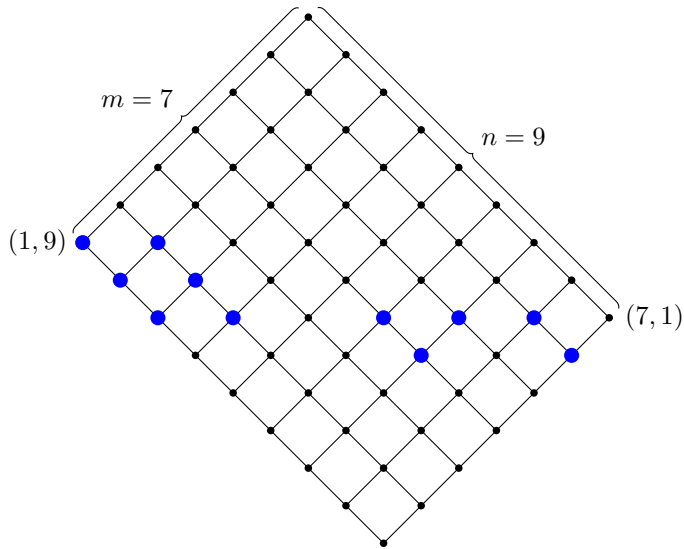
\begin{figure}[htbp]
    \centering
\begin{tikzpicture}[scale=.5]
\foreach \x in {0,...,6}
	{\foreach \y in {0,...,8}
		{\fill (\x - \y, \x + \y) circle (0.1cm) {};
	         \ifthenelse{\x < 6}
			{\draw (\x - \y, \x + \y) -- (\x - \y + 1, \x + \y + 1);}{}
		\ifthenelse{\y < 8}
			{\draw (\x - \y, \x + \y) -- (\x - \y - 1, \x + \y+1);}{}
		}
	}
\fill[blue] (5 - 0, 5 + 0) circle (0.2cm) {};
\fill[blue] (5 - 1, 5 + 1) circle (0.2cm) {};
\fill[blue] (4 - 2, 4 + 2) circle (0.2cm) {};
\fill[blue] (3 - 2, 3 + 2) circle (0.2cm) {};
\fill[blue] (3 - 3, 3 + 3) circle (0.2cm) {};
\fill[blue] (0 - 8, 0 + 8) circle (0.2cm) {};
\fill[blue] (0 - 7, 0 + 7) circle (0.2cm) {};
\fill[blue] (0 - 6, 0 + 6) circle (0.2cm) {};
\fill[blue] (1 - 7, 1 + 7) circle (0.2cm) {};
\fill[blue] (1 - 6, 1 + 6) circle (0.2cm) {};
\fill[blue] (1 - 5, 1 + 5) circle (0.2cm) {};
\draw (0 - 8, 0 + 8) node[left=.25em] {$(1, 9)$};
\draw (6 - 0, 6 + 0) node[right=.25em] {$(7, 1)$};
\draw[decoration={brace, raise=.5em},decorate]
  (0 - 8,0 + 8) -- node[above left=.5em] {$m = 7$} (6 - 8, 6 + 8);
\draw[decoration={brace, raise=.5em, mirror},decorate]
  (6 - 0,6 + 0) -- node[above right=.5em] {$n = 9$} (6 - 8, 6 + 8);
\end{tikzpicture}
    \caption{An interval-closed set of the poset $[7]\times[9]$}
    \label{fig:ex_ICS}
\end{figure}

\begin{definition}\label{def:antichain}
An \emph{antichain poset} of $m$ distinct, pairwise incomparable elements is denoted as $\mathbf{m}$. The \emph{ordinal sum of $n$ antichains} $\mathbf{a}_1\oplus\mathbf{a}_2\oplus\cdots\oplus\mathbf{a}_n$ is the poset constructed using the elements from these antichain posets with order relation $a\leq b$ whenever $a\in\mathbf{a}_i,b\in\mathbf{a}_j$ and $i\leq j$.
\end{definition}

In \cite{ELMSW}, the authors enumerated interval-closed sets of various families of posets. 
Generalizing the simple fact that the cardinality of $\IC([n])$ is $\binom{n+1}{2}+1$, they counted interval-closed sets of ordinal sums of antichains.
\begin{thm}[\protect{\cite[Thm.\ 3.3]{ELMSW}}]\label{thm:gen_ord_sum_ics_card}
The cardinality of $\IC(\mathbf{a}_1\oplus\mathbf{a}_2\oplus\cdots\oplus\mathbf{a}_n)$ is  $1+\sum_{1\leq i\leq n}(2^{a_i}-1)+\sum_{1\leq i<j\leq n}(2^{a_i}-1)(2^{a_j}-1)$.
\end{thm}

They also gave a direct enumeration of ICS in $[2]\times[n]$.
\begin{thm}[\protect{\cite[Thm.\ 4.2]{ELMSW}}]\label{prodofchainICS}
The cardinality of $\IC([2] \times [n])$ is $1+n+n^2+ \frac{n+1}{2} \binom{n+2}{3}$. 
\end{thm}

Finally, they enumerated certain ICS in $[m]\times[n]$.
\begin{thm}[\protect{\cite[Thm.\ 4.4]{ELMSW}}]\label{thm:Narayana}
The number of interval-closed sets of $[m] \times [n]$ containing at least one element of the form $(a, b)$ for each $a \in [m]$ is the Narayana number
\[ N(m+n,n) = \frac{1}{m+n}\binom{m+n}{n}\binom{m+n}{n-1}
.  \]
\end{thm}

In the next section, we study interval-closed sets of $[m]\times[n]$, interpreting them in terms of pairs of lattice paths as well as certain colored Motzkin paths; we then derive an explicit generating function for their enumeration.

\section{Interval-closed sets of rectangle posets and bicolored Motzkin paths} 
\label{sec:rectangle}
In this section, we prove Theorem~\ref{thm:A}, which gives a generating function enumerating interval-closed sets of the poset $[m]\times[n]$. We begin by giving two bijections from interval-closed sets of $[m]\times[n]$ to pairs of lattice paths. The first pair  $(L,U)$ 
consists of the \emph{upper} and \emph{lower} paths that trace out the smallest order ideal and order filter, respectively, containing an interval-closed set. We discuss this bijection and its implications in Subsection~\ref{ssec:latticepaths_rectangles}. In Subsection~\ref{ssec:bicolored} we give a bijection to the pair of paths $(B,T)$ (\emph{bottom} and \emph{top} paths) which trace out, respectively, the largest order ideal that does not contain the ICS and the smallest order ideal that does contain the ICS.
We then prove Theorem \ref{thm:Motzkin_bijection}, which uses these paths to give a bijection between $\IC([m]\times[n])$ and certain bicolored Motzkin paths. Subsection~\ref{sec:directGF} uses this bijection to prove Theorem~\ref{thm:A}.
Subsection~\ref{ssec:extracting_formulas} extracts the coefficients of this generating function for small parameter values, giving for example a formula for $\card{\IC([3]\times[n])}$. Subsection~\ref{sec:Motzkin_stats} translates statistics between interval-closed sets and Motzkin paths via the bijection of Theorem \ref{thm:Motzkin_bijection}. Finally, Subsection~\ref{sec:Bminuscule} proves Theorem~\ref{thm:B}, giving a generating function for interval-closed sets of the type $B_n$ minuscule poset, or, equivalently, vertically symmetric ICS in $[n]\times[n]$.

\subsection{A bijection to pairs of paths}
\label{ssec:latticepaths_rectangles}
In this subsection, we associate a pair of paths $(L,U)$ to each interval-closed set in $[m]\times [n]$. We then use these paths in Proposition~\ref{prop:fullNarayana} to show that certain interval-closed sets, which we call \emph{full}, are enumerated by the Narayana numbers. Finally, we characterize in Lemma~\ref{prop:paths_in_poset_language} several subsets of the poset in terms of these paths.

Denote by $\mathcal{L}_{m,n}$ the set of lattice paths in $\mathbb{R}^2$ from $(0, n)$ to $(m + n, m)$ with steps $\uu=(1,1)$ and $\dd=(1,-1)$.  It is well known that $\card{\mathcal{L}_{m,n}}=\binom{m+n}{m}$.  There is a standard bijection between order ideals of $[m]\times[n]$ and $\mathcal{L}_{m,n}$ (see e.g.,~\cite[Def.~4.14, Fig.~6]{SW2012}). This bijection proceeds by constructing, on the dual graph of the Hasse diagram, a path  that separates the order ideal from the rest of the poset. The path begins to the left of the leftmost poset element ($(1,n)$ in poset coordinates), ends to the right of the rightmost poset element ($(m,1)$ in poset coordinates), and consists of $m$ up-steps $\uu$  and $n$ down-steps $\dd$.  
(Note that the Cartesian coordinates in $\mathbb{R}^2$, which we use for the paths, are different from the coordinates that we use to refer to elements of the poset.)
A similar path may be constructed to separate an order filter from the rest of the poset. 

Given an interval-closed set $I$ of $[m] \times [n]$, let us describe how to associate a pair of lattice paths $(L,U)$ to $I$. Let $U$ be the path separating the order ideal $\oi(I)$ from the rest of the poset, and $L$ be the path separating the order filter $\f(I)$ from the rest of the poset.
Both paths begin at $\left(0,n\right)$, end at $\left(m + n,m\right)$, and consist of steps $\uu = (1, 1)$ and  $\dd = (1, -1)$.  Among all such paths, the \emph{upper path} $U$ is the lowest path 
that leaves all the elements of $I$ below it,
while the \emph{lower path} $L$ is the highest path 
that leaves all the elements of $I$ above it.
See Figure \ref{fig:UL} for an example. 

\begin{figure}[htb]
\centering
\rotatebox{45}{\begin{tikzpicture}[scale=.7]
\fill[beige] (-.25, 7.25) -- (5.25, 7.25) -- (5.25, 1.75) -- (4.75, 1.75) -- (4.75, 2.75) -- (3.75, 2.75) -- (3.75, 3.75) -- (2.75, 3.75) -- (2.75, 4.75) -- (1.75, 4.75) -- (1.75, 6.75) -- (-.25, 6.75) -- cycle;
\fill[pinkcheeks] (2, 4) circle (.35cm);
\fill[lightgray] (-.25, .75) -- (-.25, 5.25) -- (.25, 5.25) -- (.25, 4.25) -- (1.25, 4.25) --(1.25, 3.25) -- (2.25, 3.25) --(2.25, 1.25) --(4.25, 1.25) --(4.25, .75) --cycle;
\foreach \x in {0,...,5}
	{\foreach \y in {1,...,7}
		{\fill (\x, \y) circle (0.07cm) {};
	         \ifthenelse{\x < 5}
			{\draw (\x , \y) -- (\x + 1, \y);}{}
		\ifthenelse{\y < 7}
			{\draw (\x, \y) -- (\x, \y+1);}{}
		}
	}
\fill[blue] (5 , 1) circle (0.14cm) {};
\fill[blue] (4 , 2) circle (0.14cm) {};
\fill[blue] (3 , 2) circle (0.14cm) {};
\fill[blue] (3 , 3) circle (0.14cm) {};
\fill[blue] (0 , 6) circle (0.14cm) {};
\fill[blue] (1 , 6) circle (0.14cm) {};
\fill[blue] (1 , 5) circle (0.14cm) {};
\draw[very thick, realpurple, dashed] (5.5, .5) -- (5.5, 1.52) node[xshift=0.25cm, yshift=0.25cm] {\rotatebox{-45}{\large $U$}} -- (4.52, 1.52) -- (4.52, 2.5) -- (3.5, 2.5) -- (3.5, 3.5) -- (1.5, 3.5) -- (1.5, 6.5) -- (-0.48, 6.5) -- (-0.48, 7.5);
\draw[very thick, darkgreen] (5.5, .5) -- (4.48, 0.5) node[xshift=-.25cm, yshift=-.25cm]{\rotatebox{-45}{\large $L$}}  -- (4.48, 1.48) -- (2.5, 1.48) --  (2.5, 4.5) --(0.5, 4.5) -- (0.5, 5.5) -- (-.52, 5.5) -- (-0.52, 7.5);
\end{tikzpicture}}
    \caption{An interval-closed set of $P = [6]\times[7]$ (shown with the small blue dots) and its associated upper and lower paths $U$ (dashed) and $L$.  The large pink dot is the only element of $P$ incomparable with $I$, as it is below $L$ and above $U$. The order filter $\f(I)$ consists of the elements of $I$ and the elements in the beige region, whereas $\oi(I)$ consists of the elements of $I$ and the elements in the gray region.}
    \label{fig:UL}
\end{figure}
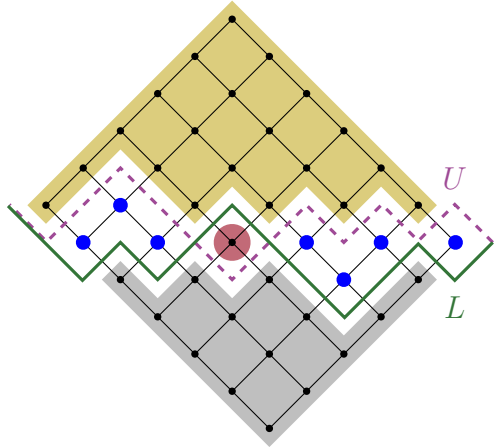

Say that $I$ is \emph{full} if $L$ and $U$ share no points other than their endpoints.  The enumeration of full interval-closed sets is closely related to Theorem~\ref{thm:Narayana}.

\begin{prop}
\label{prop:fullNarayana}
The number of full interval-closed subsets of $[m] \times [n]$ is the Narayana number 
\[ N(m+n-1,n) = \frac{1}{m + n - 1} \binom{m + n - 1}{m} \binom{m + n - 1}{n}.  \]
\end{prop}
\begin{proof}
Consider $I\in \IC([m]\times[n])$ and define a ``shift'' map $\varphi$ on the associated paths $U$ and $L$, as follows: $\varphi$ adds an up-step $\uu$ to the beginning of $U$ and an up-step $\uu$ to the end of $L$.  This results in a pair of paths $\varphi(U)=\uu U$ and $\varphi(L)=L\uu$ in the poset $[m+1]\times[n]$; see Figure \ref{fig:shiftmap} for an example. When we start with an ICS in $[m] \times [n]$ that has at least one element of the form $(a, b)$ for each $a \in [m]$, the associated path $U$ is weakly above the path $L$.  Therefore, after shifting, the new path $\varphi(U)$ is strictly above the new path $\varphi(L)$ (except at their endpoints), and so the associated ICS in $[m+1]\times[n]$ is full.

\begin{figure}[htb]
\begin{center}
\rotatebox{45}{\begin{tikzpicture}[scale=.7]
\foreach \x in {1,...,3}
	{\foreach \y in {1,...,7}
		{\fill (\x, \y) circle (0.07cm) {};
	         \ifthenelse{\x < 3}
			{\draw (\x , \y) -- (\x + 1, \y);}{}
		\ifthenelse{\y < 7}
			{\draw (\x, \y) -- (\x, \y+1);}{}
		}
	}
\fill[blue] (1, 6) circle (0.14cm) {};
\fill[blue] (1, 5) circle (0.14cm) {};
\fill[blue] (2, 4) circle (0.14cm) {};
\fill[blue] (3, 2) circle (0.14cm) {};
\fill[blue] (3, 1) circle (0.14cm) {};
\draw[realpurple, very thick, dashed] (3.5, .5) -- (3.5, 2.5) -- (2.52, 2.5) -- (2.52, 4.52) -- (1.52, 4.52) -- (1.52, 6.5) -- (.52, 6.5) -- (.52, 7.5);
\draw[darkgreen, very thick] (3.5, .5) -- (2.48, .5) -- (2.48, 3.5) -- (1.5, 3.5) -- (1.48, 4.48) -- (0.48, 4.5) -- (.48, 7.5);
 \end{tikzpicture}}
\raisebox{3cm}{$\longrightarrow$}
\rotatebox{45}{\begin{tikzpicture}[scale=.7]
\foreach \x in {1,...,4}
	{\foreach \y in {1,...,7}
		{\fill (\x, \y) circle (0.07cm) {};
	         \ifthenelse{\x < 4}
			{\draw (\x , \y) -- (\x + 1, \y);}{}
		\ifthenelse{\y < 7}
			{\draw (\x, \y) -- (\x, \y+1);}{}
		}
	}
\fill[blue] (1, 6) circle (0.14cm) {};
\fill[blue] (1, 5) circle (0.14cm) {};
\fill[blue] (2, 4) circle (0.14cm) {};
\fill[blue] (3, 2) circle (0.14cm) {};
\fill[blue] (3, 1) circle (0.14cm) {};
\draw[realpurple, very thick, dashed] (4.5, .5) -- (4.5, 2.5) -- (3.5, 2.5) -- (3.5, 4.5) -- (2.5, 4.5) -- (2.5, 6.5) -- (1.5, 6.5) -- (1.5, 7.5) -- (.5, 7.5);
\draw[darkgreen, very thick] (4.5, .5) -- (2.5, .5) -- (2.5, 3.5) -- (1.5, 3.5) -- (1.5, 4.5) -- (0.5, 4.5) -- (.5, 7.5);
\fill[cyan] (1, 7) circle (0.14cm) {};
\fill[cyan] (2, 6) circle (0.14cm) {};
\fill[cyan] (2, 5) circle (0.14cm) {};
\fill[cyan] (3, 4) circle (0.14cm) {};
\fill[cyan] (3, 3) circle (0.14cm) {};
\fill[cyan] (4, 2) circle (0.14cm) {};
\fill[cyan] (4, 1) circle (0.14cm) {}; \end{tikzpicture}}
\end{center}
    \caption{An illustration of the shift map $\varphi$ from the proof of Proposition~\ref{prop:fullNarayana}.}    \label{fig:shiftmap}
\end{figure}
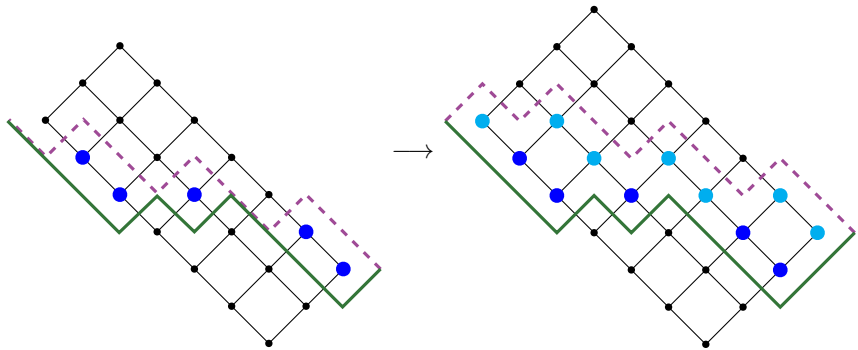

One can see that $\varphi$ is invertible, and so it is a bijection between interval-closed subsets of $[m] \times [n]$ that have at least one element of the form $(a, b)$ for each $a \in [m]$
and full interval-closed subsets of $[m + 1] \times [n]$.  The enumeration then follows from Theorem~\ref{thm:Narayana}.
\end{proof}

The paths $L$ and $U$ can also be described in poset language. We will use this lemma in Section~\ref{sec:Motzkin_stats} to translate statistics via the bijections of this paper.  An illustration of the four sets in the lemma appears in Figure~\ref{fig:UL}. 
Note we state this lemma not only for the poset $[m]\times[n]$, but also for any subposet that is itself a full interval-closed set of $[m]\times[n]$. 

\begin{lem}\label{prop:paths_in_poset_language} 
Let the poset $P$ be a full interval-closed set of $[m]\times[n]$.
    Given $I\in\IC(P)$ 
    with lower path $L$ and upper path $U$, one has the following characterization of the elements of $P$ according to their position in relation to $L$ and $U$:
    \begin{itemize}
        \item the elements above $L$ and below $U$ are exactly those in $I$,
        \item the elements below both $L$ and $U$ are exactly those in $\oi{(I)}\setminus I$,
        \item the elements above both $L$ and $U$ are  exactly those in $\f{(I)}\setminus I$, and
        \item the elements below $L$ and above $U$ are those that are incomparable with $I$.
    \end{itemize}
\end{lem}

\begin{proof}
   By definition, the elements of $P$ below $U$ are exactly those in the order ideal $\oi{(I)}$, and the elements of $P$ above $L$ are exactly those in the order filter $\f{(I)}$.
      An element $z\in P$ is in the intersection $\oi{(I)}\cap\f{(I)}$ if and only if $x\le z$ for some $x\in I$ and $z\le y$ for some $y\in I$. Since $I$ is an interval-closed set, this implies that $z\in I$. Hence, $\f{(I)} \cap \oi{(I)}= I$, proving the first three statements. 
      
    For the fourth statement, note that elements below $L$ and above $U$ are those in $P \setminus (\f{(I)} \cup \oi{(I)})$, that is, elements in $P$ that are neither larger nor smaller than any element in $I$. In other words, these are the elements that are incomparable with $I$.
\end{proof}

This perspective will be used in
\cite{LLMSW} to analyze the action of \emph{rowmotion} on interval-closed sets of $[m]\times[n]$.

\subsection{From pairs of paths to bicolored Motzkin paths}\label{ssec:bicolored}

In this subsection, we associate a slightly different pair of paths $(B,T)$ to each interval-closed set in $[m]\times [n]$ 
as an intermediate step towards a bijection between $\IC([m]\times[n])$ and certain bicolored Motzkin paths. 

As described in Section~\ref{ssec:latticepaths_rectangles}, the set of order ideals of $[m]\times[n]$ is in natural bijection with the set of lattice paths $\Lmn$ from $(0,n)$ to $(m+n,m)$ with steps $\uu$ and $\dd$. 
Let $J_1,J_2$ be order ideals of $[m]\times[n]$, and let $B,T\in\Lmn$ be their corresponding lattice paths. Then  $J_1\subseteq J_2$ if and only if $B$ lies weakly below $T$. We will write this as $B\le T$.
Let $\LLmn=\{(B,T):B,T\in\Lmn, B\le T\}$. 

Our goal is to enumerate interval-closed sets of $[m]\times[n]$.
Any interval-closed set can be expressed as $J_2\setminus J_1$ for some pair of order ideals $J_1,J_2$ such that $J_1\subseteq J_2$, and any such pair of order ideals determines an ICS. However, $J_1$ and $J_2$ are not unique in general; for example, the empty set can be written as $J\setminus J$ for any order ideal $J$.

In general, given $(B,T)\in\LLmn$, the steps where $B$ and $T$ coincide are irrelevant when determining the corresponding interval-closed set. This is because the interval-closed set has elements in the $i$th vertical ``file'' (i.e., elements $(a,b)\in[m]\times [n]$ such that $b-a=i+n-1$) if and only if the $i$th step of $B$ is strictly below the $i$th step of $T$. 
Thus, interval-closed sets of $[m]\times[n]$ are in bijection with equivalence classes of pairs $(B,T)\in\LLmn$, where the equivalence relation allows us to freely change the portions of $B$ and $T$ where these two paths coincide, as long as we preserve the portions of $B$ and $T$ that are disjoint. To enumerate these equivalence classes, let us introduce another type of lattice paths.

Denote by $\MMl$ the set of {\em bicolored Motzkin paths} of length $\ell$. These are lattice paths from $(0,0)$ to $(\ell,0)$ that never go below the $x$-axis and consist of steps of four types: $\uu=(1,1)$, $\dd=(1,-1)$, and two kinds of horizontal steps $(1,0)$, which we will denote by $\hh_1$ and $\hh_2$.
Denote by $u(M)$ the number of $\uu$ steps in $M$, and define $d(M)$, $h_1(M)$ and $h_2(M)$ similarly. Let $\MM=\bigcup_{\ell\ge0}\MMl$.

Consider the following well known bijection (see e.g.,~\cite{Elizalde-symmetry}) between $\bigcup_{m+n=\ell}\LLmn$ and $\MMl$. 
Given $(B,T)\in\LLmn$ and $\ell=m+n$, let 
$M\in\MMl$ be the path whose $i$th step $m_i$ is determined by the $i$th steps of $B$ and $T$, as follows:
\begin{equation}\label{eq:mi}
m_i=\begin{cases} \uu & \text{if $b_i=\dd$ and $t_i=\uu$},\\
\dd & \text{if $b_i=\uu$ and $t_i=\dd$},\\
\hh_1 & \text{if $b_i=\uu$ and $t_i=\uu$},\\
\hh_2 & \text{if $b_i=\dd$ and $t_i=\dd$}. \end{cases}
\end{equation}
Under this bijection, we have $(B,T)\in\LLmn$ if and only if $u(M)+h_1(M)=m$ and $d(M)+h_2(M)=n$. Let $\MM_{m,n}$ denote the set of $M\in\MM_{m+n}$ such that $u(M)+h_1(M)=m$ and $d(M)+h_2(M)=n$.

The fact that $B\le T$ guarantees that $M$ stays weakly above the $x$-axis, and that steps where $B$ and $T$ coincide correspond to horizontal steps ($\hh_1$ or $\hh_2$) of $M$ that lie on the $x$-axis.
In particular, changing steps where $B$ and $T$ coincide (while preserving the portions where $B$ and $T$ are disjoint) corresponds to rearranging the horizontal steps of $M$ within each maximal block of adjacent horizontal steps on the $x$-axis.
Thus, interval-closed sets of $[m]\times[n]$ are in bijection with equivalence classes of paths in  $\MM_{m,n}$, where the equivalence relation is given by the above rearrangements. An easy way to pick one representative from each equivalence class is to consider paths where no $\hh_2$ on the $x$-axis is immediately followed by a $\hh_1$, i.e., every block of horizontal steps on the $x$-axis is of the form $\hh_1^r\hh_2^s$ for some $r,s\ge0$. Let $\tMM$, $\tMMl$, and $\tMMmn$ respectively be the sets of paths in $\MM$, $\MMl$, and $\MMmn$ with this property. In terms of the paths $(B,T)$, this convention for picking a representative corresponds to requiring the blocks where $B$ and $T$ coincide to be of the form $\uu^r\dd^s$. In particular, the resulting path $B$ coincides with the path $L$ of the previous subsection.

The above discussion yields the following theorem.
\begin{thm}\label{thm:Motzkin_bijection}
    The set $\IC([m]\times[n])$ of interval-closed sets of $[m]\times[n]$ is in bijection with the set $\tMMmn$ of bicolored Motzkin paths where no $\hh_2$ on the $x$-axis is immediately followed by a $\hh_1$, and such that $u(M)+h_1(M)=m$ and $\dd(M)+h_2(M)=n$.   
\end{thm}

\begin{example}\label{ex:Motzkin_bijection} 
    Figure~\ref{ex_paths} shows an example of an interval-closed set of $[13] \times [14]$ with paths $T$ (in blue, dashed) and $B$ (in green) with their overlap in purple.     
    We have
    \begin{align*} T&=\dd \ \uu \ \uu \ \uu \ \dd \ \dd \ \dd \ \uu \ \uu \ \dd \ \uu \ \uu \ \uu \ \dd \ \dd \ \dd \ \uu \ \dd \ \uu \ \dd \ \uu \ \dd \ \dd \ \dd \ \uu \ \uu \ \dd,\\
    B&= \dd \ \dd \ \uu \ \dd \ \dd \ \uu \ \uu \ \uu \ \uu \ \dd \ \dd \ \uu \ \dd \ \dd \ \dd \ \uu \ \uu \ \uu \ \uu \ \dd \ \dd \ \dd \ \dd \ \uu \ \uu \ \uu \ \dd.\end{align*}
    Using (1), we obtain
    $$M = \hh_2 \ \uu \ \hh_1 \ \uu \ \hh_2 \ \dd \ \dd \ \hh_1 \ \hh_1 \ \hh_2 \ \uu \ \hh_1 \ \uu \ \hh_2 \ \hh_2 \ \dd \ \hh_1 \ \dd \ \hh_1 \ \hh_2 \ \uu \ \hh_2 \ \hh_2 \ \dd \ \hh_1 \ \hh_1 \ \hh_2,$$
    which is shown in Figure \ref{ex_motzkin_path}.
\end{example}

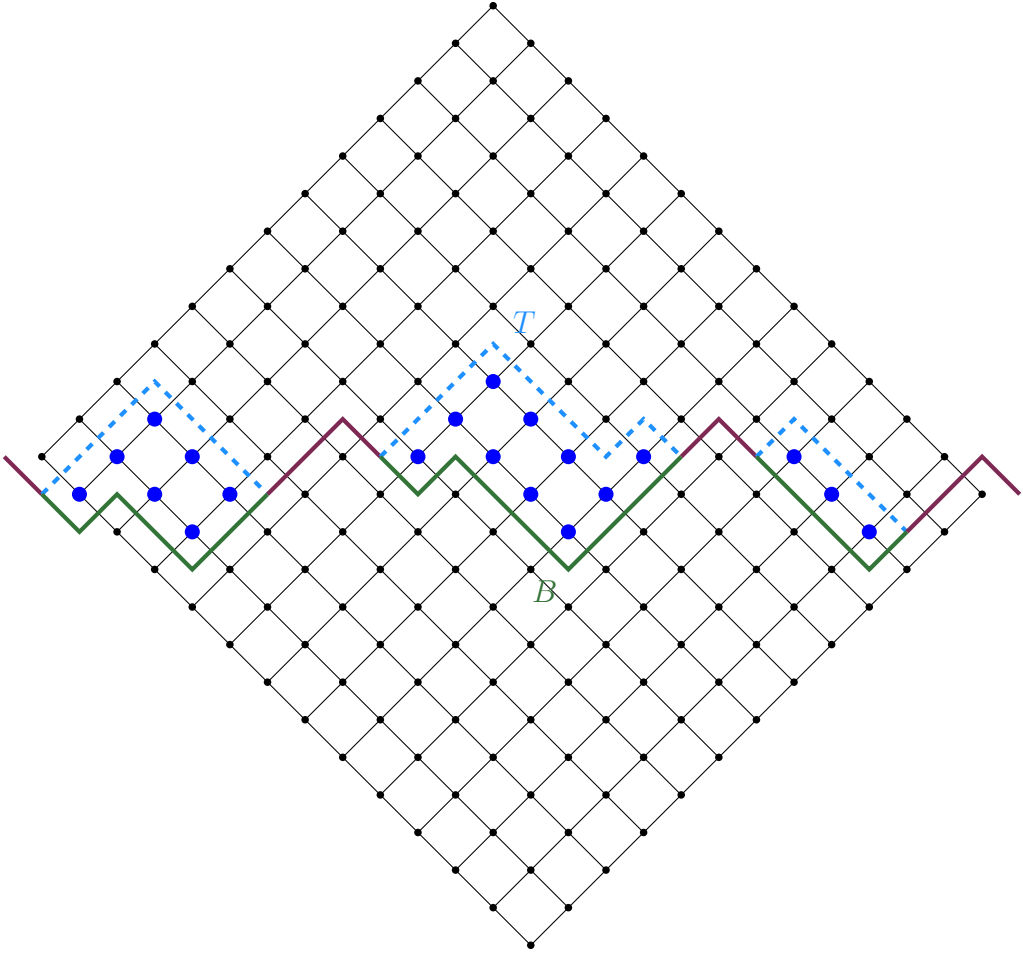
\begin{figure}[htb]
\begin{center}
\begin{tikzpicture}[scale=.5]
\foreach \x in {1,...,13}
	{\foreach \y in {1,...,14}
		{\fill (\x - \y, \x + \y) circle (0.1cm) {};
	         \ifthenelse{\x < 13}
			{\draw (\x - \y, \x + \y) -- (\x - \y + 1, \x + \y + 1);}{}
		\ifthenelse{\y < 14}
			{\draw (\x - \y, \x + \y) -- (\x - \y - 1, \x + \y+1);}{}
		}
	}
\fill[blue] (-12, 14) circle (0.2cm) {};
\fill[blue] (1 - 12, 3 + 12) circle (0.2cm) {};
\fill[blue] (2 - 12, 4 + 12) circle (0.2cm) {};
\fill[blue] (2 - 12, 2 + 12) circle (0.2cm) {};
\fill[blue] (3 - 12, 3 + 12) circle (0.2cm) {};
\fill[blue] (3 - 12, 1 + 12) circle (0.2cm) {};
\fill[blue] (4 - 12, 2 + 12) circle (0.2cm) {};
\fill[blue] (-3, 1 + 14) circle (0.2cm) {};
\fill[blue] (-2, 16) circle (0.2cm) {};
\fill[blue] (-1, 17) circle (0.2cm) {};
\fill[blue] (-1, 15) circle (0.2cm) {};
\fill[blue] (0, 16) circle (0.2cm) {};
\fill[blue] (0, 14) circle (0.2cm) {};
\fill[blue] (1, 15) circle (0.2cm) {};
\fill[blue] (1, 13) circle (0.2cm) {};
\fill[blue] (2, 14) circle (0.2cm) {};
\fill[blue] (3, 15) circle (0.2cm) {};
\fill[blue] (7, 15) circle (0.2cm) {};
\fill[blue] (8, 14) circle (0.2cm) {};
\fill[blue] (9, 13) circle (0.2cm) {};
\draw[burgundy, ultra thick] (-14, 15) -- (-13, 14);
\draw[babyblue, ultra thick, dashed] (-13, 14) -- (-10, 17) -- (-7, 14);
\draw[burgundy, ultra thick] (-7, 14) -- (-5, 16) -- (-4, 15);
\draw[babyblue, ultra thick, dashed] (-4, 15) -- (-1, 18)node[above right] {{ \large $T$}} -- (2, 15) -- (3, 16) -- (4, 15);
\draw[burgundy, ultra thick] (4, 15) -- (5, 16) -- (6, 15);
\draw[babyblue, ultra thick, dashed] (6, 15) -- (7, 16) -- (10, 13);
\draw[burgundy, ultra thick] (10, 13) -- (12, 15) -- (13, 14);
\draw[darkgreen, ultra thick] (-13, 14) -- (-12, 13) -- (-11, 14) -- (-9, 12) -- (-7, 14);
\draw[darkgreen, ultra thick] (-4, 15) -- (-3, 14) -- (-2, 15) -- (1, 12)node[below left] {{\large $B$}} -- (4, 15);
\draw[darkgreen, ultra thick] (6, 15) -- (9, 12) -- (10, 13);
\end{tikzpicture}
\end{center}
\caption{An interval-closed set in $P = [13] \times [14]$ with associated lattice paths $T$ (dashed) and $B$.}\label{ex_paths}
\end{figure}

\begin{figure}[htb]
\begin{center}
\begin{tikzpicture}[scale=.5]
 \draw[gray,thin] (0,0) grid (27,3);
   \draw (-1, -1) node {M =};
   \draw (0.5, -1) node {$\hh_2$};
   \draw (1.5, -1) node {$\uu$};
   \draw (2.5, -1) node {$\hh_1$};
   \draw (3.5, -1) node {$\uu$};
   \draw (4.5, -1) node {$\hh_2$};
    \draw (5.5, -1) node {$\dd$};
    \draw (6.5, -1) node {$\dd$};
    \draw (7.5, -1) node {$\hh_1$};
     \draw (8.5, -1) node {$\hh_1$};
      \draw (9.5, -1) node {$\hh_2$};
       \draw (10.5, -1) node {$\uu$};
     \draw (11.5, -1) node {$\hh_1$};
      \draw (12.5, -1) node {$\uu$};
       \draw (13.5, -1) node {$\hh_2$};
        \draw (14.5, -1) node {$\hh_2$};
         \draw (15.5, -1) node {$\dd$};
          \draw (16.5, -1) node {$\hh_1$};
        \draw (17.5, -1) node {$\dd$};
         \draw (18.5, -1) node {$\hh_1$};
          \draw (19.5, -1) node {$\hh_2$};
           \draw (20.5, -1) node {$\uu$};
    \draw (21.5, -1) node {$\hh_2$};
        \draw (22.5, -1) node {$\hh_2$};
     \draw (23.5, -1) node {$\dd$};
      \draw (24.5, -1) node {$\hh_1$};
       \draw (25.5, -1) node {$\hh_1$};
        \draw (26.5, -1) node {$\hh_2$};
  \draw[red, very thick] (0, 0) to[out=45, in=225, looseness=1.5]  (1, 0);
    \draw[blue, very thick] (1,0) -- (2, 1)  -- (3, 1) -- (4, 2);
      \draw[red, very thick] (4, 2)  to[out=45, in=225, looseness=1.5] (5, 2);
        \draw[blue, very thick] (5,2) -- (6, 1) -- (7, 0)  -- (8, 0)   -- (9, 0); 
 \draw[red, very thick] (9, 0) to[out=45, in=225, looseness=1.5] (10, 0); 
 \draw[blue, very thick] (10, 0) --(11, 1)  -- (12, 1) -- (13,2);
\draw[red, very thick] (13, 2)  to[out=45, in=225, looseness=1.5] (14, 2)  to[out=45, in=225, looseness=1.5] (15, 2);
\draw[blue, very thick] (15, 2)  -- (16, 1)  -- (17, 1) -- (18, 0)  -- (19, 0);
    \draw[red, very thick] (19, 0)  to[out=45, in=225, looseness=1.5] (20, 0);
    \draw[blue, very thick] (20, 0) -- (21, 1);
     \draw[red, very thick] (21, 1)  to[out=45, in=225, looseness=1.5] (22, 1)  to[out=45, in=225, looseness=1.5] (23, 1); 
    \draw[blue, very thick] (23, 1) -- (24, 0)  -- (25, 0)  -- (26, 0);
    \draw[red, very thick] (26, 0)  to[out=45, in=225, looseness=1.5] (27, 0);
 \fill[black] (0,0) circle (0.2cm) {};
\fill[black] (1,0) circle (0.2cm) {};
\fill[black] (2,1) circle (0.2cm) {};
 \fill[black] (3,1) circle (0.2cm) {};
  \fill[black] (4,2) circle (0.2cm) {};
\fill[black] (5,2) circle (0.2cm) {};
\fill[black] (6,1) circle (0.2cm) {};
\fill[black] (7,0) circle (0.2cm) {};
\fill[black] (8,0) circle (0.2cm) {};
 \fill[black] (9,0) circle (0.2cm) {};
 \fill[black] (10,0) circle (0.2cm) {};
  \fill[black] (11,1) circle (0.2cm) {};
 \fill[black] (12,1) circle (0.2cm) {};
  \fill[black] (13,2) circle (0.2cm) {};
 \fill[black] (14,2) circle (0.2cm) {};
  \fill[black] (15,2) circle (0.2cm) {};
 \fill[black] (16, 1) circle (0.2cm) {};
 \fill[black] (17,1) circle (0.2cm) {};
  \fill[black] (18,0) circle (0.2cm) {};
   \fill[black] (19,0) circle (0.2cm) {};
    \fill[black] (20,0) circle (0.2cm) {};
     \fill[black] (21,1) circle 
     (0.2cm) {};
      \fill[black] (22,1) circle (0.2cm) {};
     \fill[black] (23,1) circle (0.2cm) {};
        \fill[black] (24,0) circle (0.2cm) {};
           \fill[black] (25,0) circle (0.2cm) {};
              \fill[black] (26,0) circle (0.2cm) {};
   \fill[black] (27,0) circle (0.2cm) {};
\end{tikzpicture}
\end{center}
\caption{The bicolored Motzkin path $M\in\MM_{13,14}$, with $\hh_1$ drawn as blue and straight, and $\hh_2$ as red and curved.}
\label{ex_motzkin_path}
\end{figure}
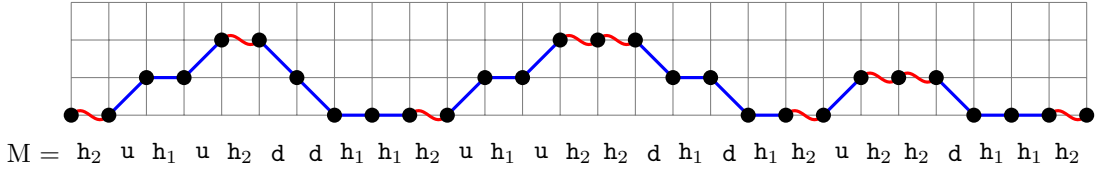

\subsection{Deriving the generating function}
\label{sec:directGF}

In this subsection, we obtain an expression for the generating function 
$$A(x,y)=\sum_{m,n\ge0} \card{\IC([m]\times[n])}\, x^m y^n$$
of interval-closed sets of $[m]\times[n]$. 

\begin{thm}\label{thm:A} The generating function of interval-closed sets of $[m]\times[n]$ is given by
    $$A(x,y)=\frac{2}{1-x-y+2xy+\sqrt{(1-x-y)^2-4xy}}.$$
\end{thm}

\begin{proof}
Using the bijection of Theorem~\ref{thm:Motzkin_bijection}, we can write
$$A(x,y)=\sum_{M\in\tMM} x^{u(M)+h_1(M)} y^{d(M)+h_2(M)}.$$
We start by recalling the derivation of the generating function for bicolored Motzkin paths,
$$C(x,y)=\sum_{M\in\MM} x^{u(M)+h_1(M)} y^{d(M)+h_2(M)},$$
as in~\cite[Lemma 2.1]{Elizalde-symmetry}. Any non-empty path in $\MM$ is either of the form $M=\hh_1M'$ or $M=\hh_2M'$, where $M'\in\MM$, or of the form $M=\uu M_1 \dd M_2$, where $M_1,M_2\in\MM$. This gives the equation
$$C(x,y)=1+(x+y)C(x,y)+xyC(x,y)^2,$$ from which we conclude
\begin{equation}\label{eq:C}
C(x,y)=\frac{1-x-y-\sqrt{(1-x-y)^2-4xy}}{2xy}.
\end{equation}

We now give a similar decomposition for non-empty paths in $\tMM$. Paths that start with a horizontal step must be of the form $M=\hh_1M'$, where $M'\in\tMM$, or $M=\hh_2M'$, where $M'$ is any path in $\tMM$ that does not start with $\hh_1$. Paths that start with an up-step are of the form $M=\uu M_1\dd M_2$, where $M_1\in\MM$ and $M_2\in\tMM$.
This decomposition yields the equation
$$A(x,y)=1+xA(x,y)+y(A(x,y)-xA(x,y))+xyC(x,y)A(x,y),$$ from which we conclude
$$
A(x,y)=\frac{1}{1-x-y+xy-xyC(x,y)}=\frac{2}{1-x-y+2xy+\sqrt{(1-x-y)^2-4xy}}.\qedhere
$$
\end{proof}

Equation~\eqref{eq:C} gives an alternative proof of Proposition~\ref{prop:fullNarayana}: via the bijection in Section~\ref{ssec:bicolored}, full interval-closed sets of $[m]\times[n]$ correspond to pairs $(B,T)$ where $B$ and $T$ only touch at their endpoints, which in turn correspond to bicolored Motzkin paths that only touch the $x$-axis at their endpoints. These are paths of the form $\uu M\dd$, where $M\in\MM$, and so their generating function is $$xy\,C(x,y)=\frac{1-x-y-\sqrt{(1-x-y)^2-4xy}}{2}.$$
The coefficient of $x^my^n$ in this generating function is $N(m+n-1,n)$, recovering
Proposition~\ref{prop:fullNarayana}. 

\subsection{Extracting formulas for small parameter values} 
\label{ssec:extracting_formulas}

From the expression in Theorem~\ref{thm:A}, one can obtain generating functions counting interval-closed sets of $[m]\times [n]$ where one of the parameters is fixed.
For example, differentiating twice with respect to $x$, we have
$$
\frac{\partial^2 A(x,y)}{\partial x^2}=\sum_{m\ge2,n\ge0} m(m-1)\card{\IC([m]\times[n])}\, x^{m-2} y^n.
$$
Setting $x=0$ and using Theorem~\ref{thm:A}, we get
$$\sum_{n\ge0} \card{\IC([2]\times[n])}\, y^n=\frac{1}{2} \left.\frac{\partial^2 A(x,y)}{\partial x^2}\right|_{x=0}=\frac{1-y+3y^2-2y^3+y^4}{(1-y)^5}.$$
Extracting the coefficient of $y^n$ gives
$$\card{\IC([2]\times[n])}=\binom{n+4}{4}-\binom{n+3}{4}+3\binom{n+2}{4}-2\binom{n+1}{4}+\binom{n}{4}=\frac{n^4+4n^3+17n^2+14n+12}{12},$$
recovering Theorem~\ref{prodofchainICS}.

Similarly, we have
$$\sum_{n\ge0} \card{\IC([3]\times[n])}\, y^n=\frac{1}{6} \left.\frac{\partial^3 A(x,y)}{\partial x^3}\right|_{x=0}=\frac{1+5y^2-5y^3+6y^4-3y^5+y^6}{(1-y)^7},$$
from where we obtain the following.
\begin{cor}
\label{cor:3xncor}
The cardinality of $\IC([3]\times[n])$ is
$$\frac{n^{6}+9 n^{5}+61 n^{4}+159 n^{3}+370 n^{2}+264 n +144}{144}.$$
\end{cor}

In general, for any fixed $m$, we have
$$\sum_{n\ge0} \card{\IC([m]\times[n])}\, y^n=\frac{1}{m!} \left.\frac{\partial^m A(x,y)}{\partial x^m}\right|_{x=0},$$
which is a rational generating function, since the square roots in the partial derivatives of $A(x,y)$ disappear when setting $x=0$. Extracting the coefficient of $y^n$ gives an expression for $\IC([m]\times[n])$, which, according to our computations for $m\le10$, seems to be a polynomial in $n$ of degree $2m$ with non-negative coefficients.

\subsection{Translating statistics between interval-closed sets and bicolored Motzkin paths}
\label{sec:Motzkin_stats}
We now translate some statistics between interval-closed sets and bicolored Motzkin paths, via the bijection of Theorem~\ref{thm:Motzkin_bijection}. See Example~\ref{ex:stats} below.
\begin{thm}
\label{thm:Motzkin_stats_bijection}
Let $I\in\IC([m]\times[n])$, and let $M\in\tMMmn$ be its image under the bijection of Theorem~\ref{thm:Motzkin_bijection}. Then,
\begin{enumerate}[label=(\alph*)]
\item the cardinality of $I$ is the area under $M$ and above the $x$-axis;
\item the number of elements of $[m]\times[n]$ that are incomparable with $I$ is equal to $\sum \#\hh_1\, \#\hh_2$, where the sum is over all maximal runs of horizontal steps of $M$ at height $0$, and $\#\hh_1$ and $\#\hh_2$ denote the number of $\hh_1$ and $\hh_2$ steps in each such run; and
\item the number of connected components of $I$ is the number of returns of $M$ to the $x$-axis. 
\end{enumerate}
\end{thm}

\begin{proof}
Let $B$ and $T$ be the lattice paths obtained from $I$ using the bijection from Subsection~\ref{ssec:bicolored}. 
Let $(i, \beta_i)$ and $(i, \tau_i)$ be the coordinates of the vertices of $B$ and $T$ after $i$ steps, respectively.
Since the paths start at $(0,n)$ and consist of steps $(1,1)$ and $(1,-1)$, we have $i+\beta_i\equiv i+\tau_i\equiv n \bmod 2$. 
Note that, in the Cartesian coordinates that we use for lattice paths, the points $(i,j)\in\mathbb{R}^2$ satisfying $i+j \equiv n+1\bmod 2$ correspond to the elements of the poset $[m]\times [n]$, provided that they are inside the rectangle with vertices $(0,n)$, $(m,n+m)$, $(m+n, m)$ and $(n,0)$. This is shown in Figure~\ref{fig:ICS_coordinates}.
\begin{figure}[htbp]
    \centering

\begin{tikzpicture}[scale=.5]
\foreach \x in {0,1,6}
	{\foreach \y in {0,1,7,8}
		{\fill (\x - \y, \x + \y) circle (0.1cm) {};
	         \ifthenelse{\x < 6}
			{\draw (\x - \y, \x + \y) -- (\x - \y + 1, \x + \y + 1);}{}
		\ifthenelse{\y < 8}
			{\draw (\x - \y, \x + \y) -- (\x - \y - 1, \x + \y+1);}{}
		}
	}

\draw[dotted] (-1,1) --(-6,6);
\draw[dotted] (0,2) --(-5,7);
\draw[dotted] (5,7) --(0,12);
\draw[dotted] (-6,10) --(-3,13);
\draw (-3,13) --(-2,14);
\draw[dotted] (-5,9)--(-2,12);
\draw (-2,12) --(-1,13);
\draw[dotted] (1,3) --(4,6);
\draw (4,6) --(5,7);
\draw[dotted] (2,2) --(5,5);
\draw (5,5) --(6,6);

\draw (-7,7) --(-6,6);
\draw (-6,8) --(-5,7);
\draw (-1,13) --(0,12);

\fill[blue] (-7, 8) circle (0.35cm) {};
\draw (0 - 8, 0 + 8) node[left=.25em] {$(0, n)$};
\draw (6 - 0, 6 + 0) node[right=.25em] {$(m+n,m)$};
\draw (- 6, 8) node[right=.25em] {$(2,n)$};
\draw (- 7, 7) node[below left] {$(1,n-1)$};
\draw (- 7, 9) node[above left] {$(1,n+1)$};
\draw (0,0) node[below=.25em] {$(n,0)$};
\draw (-1,14) node[above=.25em] {$(m, m+n)$};

\end{tikzpicture}
    \caption{The Cartesian coordinates used for lattice paths in $\Lmn$. The blue circle is element $(1,n)$ of the poset $[m]\times[n]$.}
    \label{fig:ICS_coordinates}
\end{figure}
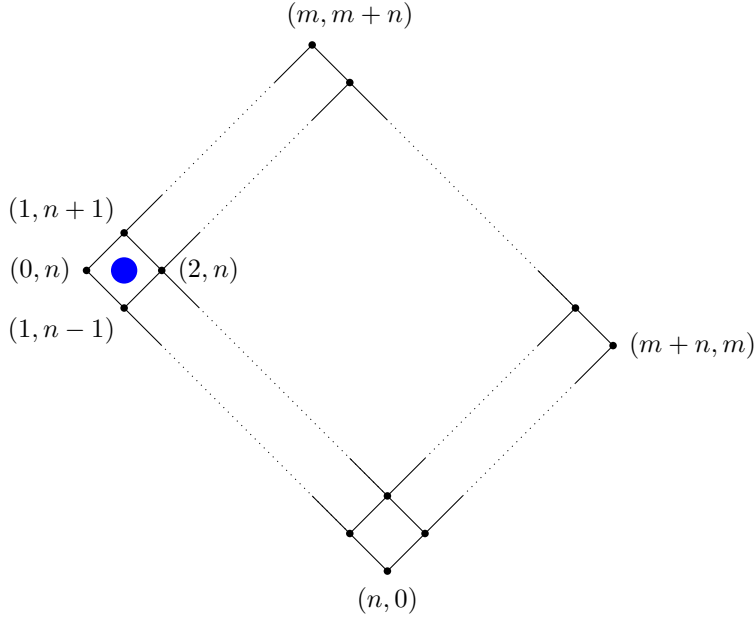

Let $d_i(B,T)=\tau_i-\beta_i$ be the distance between the paths $B$ and $T$ after $i$ steps. Since $T$ is weakly above $B$, we can write $d_i(B,T) = 2k$ for some nonnegative integer $k$, which is equal to the difference between the number of $\uu$ steps of $B$ and $T$ within their first $i$ steps. 
In the corresponding bicolored Motzkin path $M$, constructed from $B$ and $T$ using equation~\eqref{eq:mi}, this difference $k$ is equal to the number of $\uu$ steps (which occur in positions where $T$ has a $\uu$ step but $B$ does not) minus the number of $\dd$ steps (which occur in positions where $B$ has a $\uu$ step but $T$ does not) within the first $i$ steps of $M$, which in turn equals the height of $M$ after $i$ steps.
Summing over $i$, it follows that the area under $M$
(defined as the number of full squares plus half the number of triangles under $M$) and above the $x$-axis is equal to $\frac{1}{2} \sum_i d_i(B,T)$.

Let us now show that $\frac{1}{2}d_i(B,T)$ is also equal to the number of elements of $I$ with coordinates of the form $(i, j)$ in the lattice path coordinate system, for each fixed $i$. 
After $i$ steps, if $T$ is strictly above $B$, the points $(i, \beta_i +1), (i, \beta_i+3), \ldots, (i, \tau_i-1)$ are the elements of the poset above $B$ and below $T$, so they are exactly the elements of $I$ of the form $(i,j)$. There are $\frac{1}{2}d_i(B,T) = \frac{1}{2}(\tau_i-\beta_i)$ of them. Summing over $i$, we obtain $|I|=\frac{1}{2} \sum_i d_i(B,T)$. This proves part~(a).

By the last part of Lemma~\ref{prop:paths_in_poset_language}, the elements that are incomparable with $I$ are those that lie below the path $L$ and above the path $U$, defined in Subsection~\ref{ssec:latticepaths_rectangles}. 
When $L$ is below $U$, the paths $B$ and $T$ coincide with $L$ and $U$, respectively.
On the other hand, when $L$ is above $U$, then $B$ and $T$ coincide with each other, and they also coincide with $L$; in these portions, $M$ consists of horizontal steps at height $0$. Consider a maximal block where $L$ is above $U$, or equivalently, where $B$ and $T$ coincide. Let $\#\hh_1$ and $\#\hh_2$ denote the number of $\hh_1$ and $\hh_2$ steps in this block, which equals the number of $\uu$ and $\dd$ steps of $L$, respectively. In this block, 
the number of elements in $[m]\times [n]$ below $L$ and above $U$ (hence incomparable with $I$) forms an $\#\hh_1\times\#\hh_2$ rectangle, and so it contributes $\#\hh_1\,\#\hh_2$ elements. Summing over all the blocks where $L$ is above $U$, we obtain part~(b).

Finally, the connected components in $I$ correspond to the maximal blocks where $B$ is strictly below $T$, or equivalently, $M$ is strictly above the $x$-axis. Thus, the number of connected components is the number of returns of $M$ to the $x$-axis, i.e., $\dd$ steps that end at height~$0$, proving part~(c).
\end{proof}

\begin{example}
\label{ex:stats}
        Continuing Example~\ref{ex:Motzkin_bijection}, we note that the cardinality of the interval-closed set $I$ in Figure~\ref{ex_paths} is 20, which equals the area under the associated bicolored Motzkin path $M$ in Figure~\ref{ex_motzkin_path}. The number of components of $I$ is 3, which equals the number of returns of $M$. And the number of elements of the poset incomparable with $I$ is 5, which equals $\sum \#\hh_1\ \#\hh_2=2\cdot 1+1\cdot1+2\cdot1$.
\end{example}

\subsection{Counting interval-closed sets of the type $B$ minuscule poset
}
\label{sec:Bminuscule}
The product of chains poset $[m]\times[n]$ may be interpreted as a \emph{minuscule poset} associated to the type $A_{m+n-1}$ Dynkin diagram. One can study interval-closed sets of other minuscule posets. The next simplest is the type  $B_n$ minuscule poset, which is the triangular poset constructed as half of the $[n]\times[n]$ diamond poset, as seen in Figure~\ref{fig:B_minuscule}. We denote this poset by $\fB_n$. See~\cite[Section 3]{Okada21} for  background on minuscule posets.
As illustrated in Figure~\ref{fig:B_minuscule}, interval-closed sets of  $\fB_n$ are in bijection with vertically-symmetric interval-closed sets of $[n]\times[n]$.
\begin{figure}[htbp]
\begin{center}
\begin{tikzpicture}[scale=.5]
\foreach \y in {0,...,8}
	{\foreach \x in {0,...,\y}
		{\fill (\x - \y, \x + \y) circle (0.1cm) {};
	         \ifthenelse{\x < \y}
			{\draw (\x - \y, \x + \y) -- (\x - \y + 1, \x + \y + 1);}{}
		\ifthenelse{\y < 8}
			{\draw (\x - \y, \x + \y) -- (\x - \y - 1, \x + \y + 1);}{}
		}
	}
\fill[blue] (-2, 8) circle (0.2cm) {};
\fill[blue] (-1, 9) circle (0.2cm) {};
\fill[blue] (-4, 8) circle (0.2cm) {};
\fill[blue] (-5, 7) circle (0.2cm) {};
\fill[blue] (-6, 8) circle (0.2cm) {};
\fill[blue] (-6, 6) circle (0.2cm) {};
\fill[blue] (-7, 7) circle (0.2cm) {};
\fill[blue] (-8, 8) circle (0.2cm) {};
\fill[blue] (0, 8) circle (0.2cm) {};
\draw[<->] (1,8)--(2,8);
\end{tikzpicture}\quad
\begin{tikzpicture}[scale=.5]
\foreach \y in {0,...,8}
	{\foreach \x in {0,...,8}
		{\fill (\y - \x, \x + \y) circle (0.1cm) {};
	         \ifthenelse{\x < 8}
			{\draw (\y - \x, \x + \y) -- (\y - \x - 1, \x + \y + 1);}{}
		\ifthenelse{\y < 8}
			{\draw (\y - \x, \x + \y) -- (\y - \x + 1, \x + \y + 1);}{}
		}
	}
\fill[blue] (2, 8) circle (0.2cm) {};
\fill[blue] (1, 9) circle (0.2cm) {};
\fill[blue] (4, 8) circle (0.2cm) {};
\fill[blue] (5, 7) circle (0.2cm) {};
\fill[blue] (6, 8) circle (0.2cm) {};
\fill[blue] (6, 6) circle (0.2cm) {};
\fill[blue] (7, 7) circle (0.2cm) {};
\fill[blue] (8, 8) circle (0.2cm) {};
\fill[blue] (-2, 8) circle (0.2cm) {};
\fill[blue] (-1, 9) circle (0.2cm) {};
\fill[blue] (-4, 8) circle (0.2cm) {};
\fill[blue] (-5, 7) circle (0.2cm) {};
\fill[blue] (-6, 8) circle (0.2cm) {};
\fill[blue] (-6, 6) circle (0.2cm) {};
\fill[blue] (-7, 7) circle (0.2cm) {};
\fill[blue] (-8, 8) circle (0.2cm) {};
\fill[blue] (0,8) circle (0.2cm) {};
\end{tikzpicture}
\end{center}
    \caption{An interval-closed set of the minuscule poset $\fB_9$ and its corresponding symmetric interval-closed set in $[9]\times[9]$.}
    \label{fig:B_minuscule}
\end{figure}
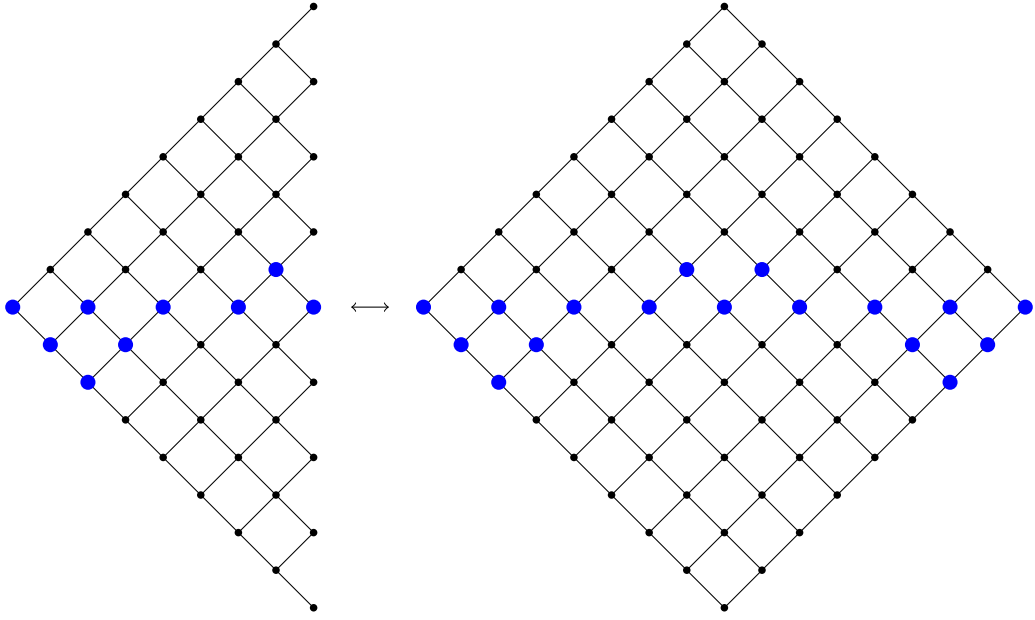

In this subsection, we find the generating function of $\IC(\fB_n)$.
\begin{thm}\label{thm:B} The generating function of interval-closed sets of $\fB_n$ is given by
    $$\sum_{n\ge0} \card{\IC(\fB_n)}\, x^n=\frac{4-10x+8x^2}{2-11x+14x^2-8x^3-(2-3x)\sqrt{1-4x}}.$$
\end{thm}

\begin{proof}
Denote by $\SSn$ the set of bicolored Motzkin paths $M\in\MMn$ that have the following symmetry: for all $i\in[2n]$, the $i$th step of $M$ is a $\uu$ if and only if the $(2n+1-i)$th step is a $\dd$, and the $i$th step is a $\hh_1$ if and only if the $(2n+1-i)$th step is a $\hh_2$.

When restricted to symmetric interval-closed sets of $[n]\times[n]$, the bijection from Section~\ref{ssec:bicolored} becomes a bijection between $\IC(\fB_n)$ and equivalence classes of paths in $\SSn$, where two paths are equivalent if one is obtained from the other by rearranging steps within the maximal blocks of $\hh_1$ and $\hh_2$ steps that lie on the $x$-axis. A canonical representative of each class can be picked by considering paths where no $\hh_2$ on the $x$-axis is immediately followed by a $\hh_1$. Let $\tSSn$ be the set of paths in $\SSn$ with this property. The above bijection implies that $\card{\IC(\fB_n)}=|\tSSn|$.

By the symmetry property, paths in $\SSn$ are uniquely determined by their left halves, which are lattice paths with $n$ steps $\uu,\dd,\hh_1,\hh_2$, starting at $(0,0)$ and never going below the $x$-axis. 
Left halves of paths in $\tSSn$ are those that do not end with an $\hh_2$ on the $x$-axis, and 
where no $\hh_2$ on the $x$-axis is immediately followed by a $\hh_1$. 

Let $\SS=\bigcup_{n\ge0}\SSn$, $\tSS=\bigcup_{n\ge0}\tSSn$, and define the generating functions $S(x)=\sum_{n\ge0} |\SSn| x^n$, and 
$$\widetilde{S}(x)=\sum_{n\ge0} |\tSSn| x^n=\sum_{n\ge0} \card{\IC(\fB_n)} x^n.$$
We obtain $S(x)$ by observing that if $H$ is the left half of a non-empty path in $\SS$, then $H=\hh_1H'$, $H=\hh_2H'$, $H=\uu H'$ (if $H$ does not return to the $x$-axis), or $H=\uu M\dd H'$ (if it does return), where $H'$ is the left half of a path in $\SS$, and $M\in\MM$. This gives the equation
$$S(x)=1+3xS(x)+x^2C(x,x)S(x),$$ 
where $C(x,y)$ is given by equation~\eqref{eq:C}. It follows that 
\begin{equation} \label{eq:S} S(x)=\frac{2}{1-4x+\sqrt{1-4x}}.\end{equation}

Similarly, if $H$ is the left half of a non-empty path in $\tSS$, then $H=\hh_1H'$, $H=\hh_2H''$, $H=\uu H'''$, or $H=\uu M\dd H'$, where $H'$ is the left half of a path in $\tSS$, $H''$ is the left half of a nonempty path in $\tSS$ that does not start with $\hh_1$, $H'''$ is the left half of a path in $\SS$, and $M\in\MM$. This decomposition yields the equation
$$\widetilde{S}(x)=1+x\widetilde{S}(x)+x(\widetilde{S}(x)-x\widetilde{S}(x)-1)+xS(x)+x^2C(x,x)\widetilde{S}(x).$$
Solving for $\widetilde{S}(x)$ and using equations~\eqref{eq:C} and~\eqref{eq:S}, we obtain
$$\widetilde{S}(x)=\frac{4-10x+8x^2}{2-11x+14x^2-8x^3+(2-3x)\sqrt{1-4x}},$$ as desired.
\end{proof}

Extracting coefficients, we see that the values of $\card{\IC(\fB_n)}$ for $1\le n\le 10$ are $$2,7,26,96,356,1331,5014,19006,72412,277058.$$

\section{Interval-closed sets of triangular posets and quarter-plane walks}
\label{sec:TypeAroot}
Another poset with a notable set of order ideals is the poset of positive roots associated to the type $A$ Dynkin diagram. This poset is triangular, as shown in Figure~\ref{fig:A14}. We denote this poset with $n$ minimal elements as $\bA_{n}$. See \cite[Section 4.6]{BjornerBrenti} for background on root posets. 
This poset is itself an interval-closed set of $[n]\times[n]$, and we can associate to each ICS in $\bA_{n}$ the same pair of lattice paths $B$ and $T$ as in the prior section, restricted to the triangular grid.

\begin{figure}[htbp]
\begin{center}
\begin{tikzpicture}[scale=.5]
\foreach \y in {0,...,13}
	{\foreach \x in {0,...,\y}
		{\fill[black] (\x + \y, \y - \x) circle (0.1cm) {};
	         \ifthenelse{\x < \y}
			{\draw[black] (\x + \y, \y - \x) -- (\x + \y + 1, \y - \x - 1);}{}
		\ifthenelse{\y < 13}
			{\draw[black] (\x + \y, \y - \x) -- (\x + \y + 1, \y - \x+1);}{}
		}
	}
\fill[blue] (2, 0) circle (0.2cm) {};
\fill[blue] (3, 1) circle (0.2cm) {};
\fill[blue] (4, 0) circle (0.2cm) {};
\fill[blue] (5, 1) circle (0.2cm) {};
\fill[blue] (8, 2) circle (0.2cm) {};
\fill[blue] (9, 3) circle (0.2cm) {};
\fill[blue] (10, 2) circle (0.2cm) {};
\fill[blue] (11, 1) circle (0.2cm) {};
\fill[blue] (12, 2) circle (0.2cm) {};
\fill[blue] (13, 1) circle (0.2cm) {};
\fill[blue] (16, 2) circle (0.2cm) {};
\fill[blue] (17, 3) circle (0.2cm) {};
\fill[blue] (18, 2) circle (0.2cm) {};
\fill[blue] (18, 4) circle (0.2cm) {};
\fill[blue] (19, 3) circle (0.2cm) {};
\fill[blue] (23, 3) circle (0.2cm) {};
\fill[blue] (24, 2) circle (0.2cm) {};
\fill[blue] (25, 1) circle (0.2cm) {};
  \draw[ultra thick, babyblue, dashed] (1, 2.25) node { \large $T$};
\draw[ultra thick, babyblue, dashed] (22, 3) -- (23.03, 4.03) -- (26, 1);
\draw[ultra thick, babyblue, dashed]  (1, 0) -- (3, 2) -- (4, 1) --(5, 2) -- (6,1);
\draw[ultra thick, babyblue, dashed]
(7,2) -- (9, 4) -- (11, 2) -- (12, 3) -- (14, 1); 
\draw[ultra thick, babyblue, dashed]
(15,2) -- (18, 5) -- (20, 3);
\draw[very thick, burgundy] (-2, -1) -- (0, 1) -- (1,0);
 \draw[ultra thick, darkgreen] (6, -1) node { \large $B$};
\draw[ultra thick, darkgreen]
(1,0) -- (2, -1) -- (3, 0) -- (4, -1) -- (6,1);
\draw[ultra thick, darkgreen]  (7, 2) -- (8, 1) -- (9, 2) -- (11, 0) -- (12, 1) -- (13, 0) -- (14,1);
\draw[ultra thick, darkgreen]
(15, 2) -- (16, 1) -- (17, 2) -- (18, 1) -- (20, 3);
\draw[ultra thick, darkgreen]
(22,3) -- (25, 0) -- (26, 1);
\draw[burgundy, very thick] (20, 3) -- (21,4) -- (22,3);
\draw[burgundy, very thick] (14,1) -- (15,2);
\draw[burgundy, very thick] (6, 1) -- (7,2);
\draw[burgundy, very thick] (26, 1) -- (28, -1);
\end{tikzpicture}
\end{center}
    \caption{An interval-closed set of the poset $\mathbf{A}_{14}$ and the corresponding pair of paths $B$ and $T$.}
    \label{fig:A14}
\end{figure}
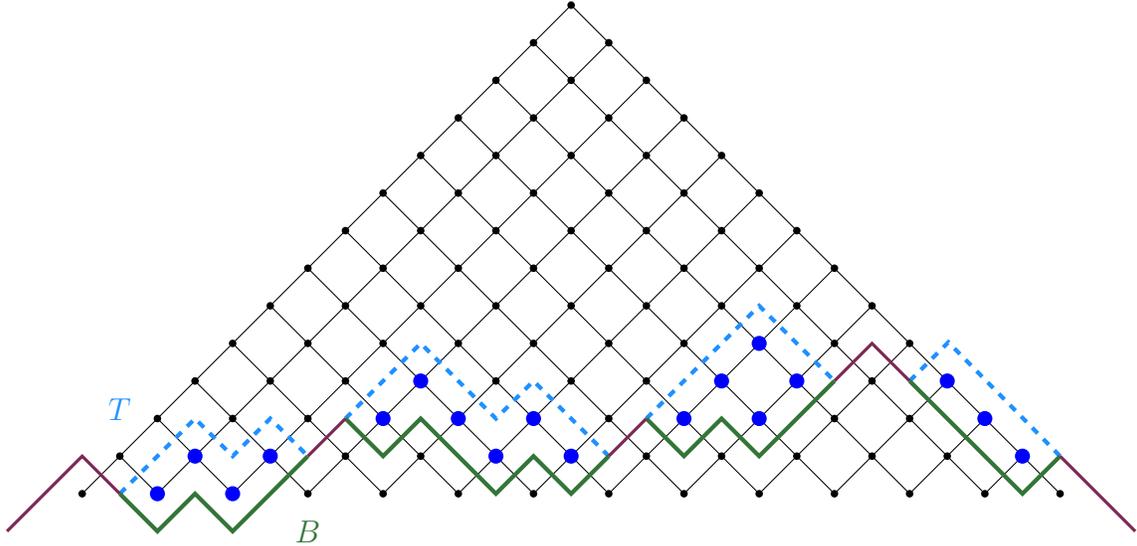

Note that $\bA_{n}$ may be obtained from $[n]\times[n]$ by truncating the bottom $n-1$ ranks. This motivates the extension of our enumeration from $\bA_n$ to the class of \emph{truncated rectangles}. For nonnegative integers $d, m, n$ with $d < \min(m, n)$, let $P_{m\times n; d}$ be the poset obtained by truncating the bottom $d$ ranks from $[m] \times [n]$. See the left of Figure~\ref{fig:truncated} for an example.
So we have $\bA_n=P_{n\times n; n-1}$.

This section studies interval-closed sets of $\bA_n$ and $P_{m\times n; d}$. In Subsection~\ref{sec:functional_equation}, we prove Theorem~\ref{thm:walks_bijection},  giving a bijection from $\IC(\bA_{n-1})$ to certain quarter-plane walks, and Theorem~\ref{thm:ICAn}, yielding a functional equation for the generating function. Subsection~\ref{sec:typeB} proves Theorem~\ref{thm:BrootGF} giving a functional equation for the generating function of  vertically symmetric interval-closed sets in $\bA_{n}$, which we then apply to the type $B$ root poset. Subsection~\ref{ssec:truncated_rectangles} proves Theorems \ref{thm:walks_bijection_truncated} and \ref{thm:ICP}, giving analogous results for truncated rectangle posets $P_{m\times n; d}$. Subsection~\ref{sec:quarter_plane_stats} translates statistics between interval-closed sets and quarter-plane walks via the bijections of Theorems~\ref{thm:walks_bijection} and \ref{thm:walks_bijection_truncated}.

\subsection{A functional equation to count interval-closed sets in the type $A$ root poset}
\label{sec:functional_equation}

Order ideals of $\bA_{n-1}$ are in bijection with Dyck paths of length $2n$, that is, lattice paths from $(0,0)$ to $(2n,0)$ with steps $\uu$ and $\dd$ that do not go below the $x$-axis. As in the bijection described in Section~\ref{sec:rectangle}, we associate each order ideal to the path that separates it from the rest of the poset, but here it will be more convenient to consider the origin as the starting point of the paths (see, e.g., Catalan objects 25 and 178 in \cite[Ch.~2]{Stanley_Catalan}). 
Denote the set of Dyck paths of length $2n$ by $\Dn$, and let $\DDn=\{(B,T):B,T\in\Dn, B\le T\}$, the set of pairs of nested Dyck paths.

Similarly to Section~\ref{ssec:bicolored}, enumerating ICS of $\bA_{n-1}$ is equivalent to enumerating pairs $(B,T)\in\DDn$ up to the equivalence relation that allows us to change the portions of $B$ and $T$ where these two paths coincide. One can canonically pick a representative of each equivalence class by requiring that, in each maximal block of steps where $B$ and $T$ coincide, up-steps come before down-steps. Note that this choice guarantees that the paths do not go below the $x$-axis.

Next we describe a bijection between $\DDn$ and the set $\Wo_{2n}$ of lattice walks in the first quadrant $\{(x,y):x,y\ge0\}$, consisting of $2n$ steps from among $\e=(1,0),\w=(-1,0),\se=(1,-1),\nw=(-1,1)$, and starting and ending at the origin.

Given $(B,T)\in\DDn$, let 
$W\in\Wo_{2n}$ be the walk  whose $i$th step $w_i$ is determined by the $i$th steps of $B$ and $T$ as follows:
\begin{equation}\label{eq:wi}
w_i=\begin{cases} \nw & \text{if $b_i=\dd$ and $t_i=\uu$},\\
\se & \text{if $b_i=\uu$ and $t_i=\dd$},\\
\e & \text{if $b_i=\uu$ and $t_i=\uu$},\\
\w & \text{if $b_i=\dd$ and $t_i=\dd$}. \end{cases}
\end{equation}

Note that the condition that $B$ does not go below the $x$-axis guarantees that $W$ stays in $x\ge0$, and the condition that $B\le T$ guarantees that $W$ stays in $y\ge0$.
Steps where $B$ and $T$ coincide correspond to steps in $W$ that lie on the $x$-axis. Thus, the above canonical representatives of the equivalence classes in $\DDn$ correspond to walks in $\Wo_{2n}$ where no $\w$ step on the $x$-axis is immediately followed by an $\e$ step.
Let $\tWo_{2n}\subseteq\Wo_{2n}$ be the subset of walks with this property.
This discussion yields the following theorem.

\begin{thm}
\label{thm:walks_bijection}
    The set $\IC(\bA_{n-1})$ of interval-closed sets of $\bA_{n-1}$ is in bijection with the set $\tWo_{2n}$ of lattice walks in the first quadrant starting and ending at the origin, and consisting of $2n$ steps from the set $\{ \e,\w,\se,\nw \}$ where no $\w$ step on the $x$-axis is immediately followed by an $\e$ step.   
\end{thm}

An example illustrating Theorem \ref{thm:walks_bijection} is given below.

\begin{example}
 Figure \ref{ex_typeA}, left, shows an interval-closed set of $\mathbf{A}_{5}$ with paths $T$ (in blue, dashed), $B$ (in green), and where they coincide in purple. We have
\begin{align*}T &= \uu \ \uu \ \uu \ \dd \ \dd \ \uu \ \uu \ \dd \ \uu \ \dd \ \dd \ \dd,\\
B & = \uu \ \uu \ \dd \ \dd \ \uu \ \uu \ \uu \ \dd \ \dd \ \uu \ \dd \ \dd.
\end{align*}
Using \eqref{eq:wi}, we form the quarter-plane walk
$$W = \e \ \e \ \nw \ \w \ \se \ \e \ \e \ \w \ \nw \ \se \ \w \ \w,$$
which is shown in Figure \ref{ex_typeA}, right.
\end{example}

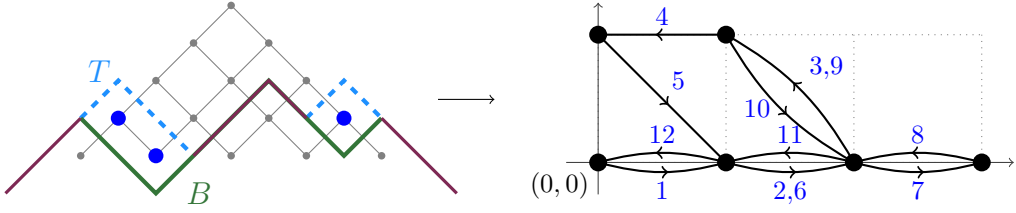
\begin{figure}[htbp]
\begin{center}
\begin{tikzpicture}[scale=.5]
\foreach \y in {0,...,4}
	{\foreach \x in {0,...,\y}
		{\fill[gray] (\x + \y, \y - \x) circle (0.1cm) {};
	         \ifthenelse{\x < \y}
			{\draw[gray] (\x + \y, \y - \x) -- (\x + \y + 1, \y - \x - 1);}{}
		\ifthenelse{\y < 4}
			{\draw[gray] (\x + \y, \y - \x) -- (\x + \y + 1, \y - \x+1);}{}
		}
	}
\fill[blue] (2, 0) circle (0.2cm) {};
\fill[blue] (1, 1) circle (0.2cm) {};
\fill[blue] (7, 1) circle (0.2cm) {};
\draw[ultra thick, babyblue, dashed] (0,1) -- (1, 2) -- (3,0);
\draw[ultra thick, babyblue, dashed] 
(6,1) -- (7, 2) -- (8,1);
\draw[ultra thick, darkgreen] (0, 1) --(2, -1) -- (5, 2) -- (7,0) -- (8,1);
 \draw[darkgreen] (3.15, -1) node {\large $B$};
  \draw[ultra thick, babyblue] (.5, 2.25) node { \large $T$};
  \draw[very thick, burgundy] (-2, -1) -- (0,1);
    \draw[very thick, burgundy] (8,1) -- (10, -1);
    \draw[very thick, burgundy] (3, 0) -- (5, 2) -- (6, 1); 
    \draw[->] (9.5,1.5)--(11,1.5);
\end{tikzpicture}\quad
\begin{tikzpicture}[scale=1.7,decoration={
    markings,
    mark=at position 0.55 with {\arrow{>}}}]
 \draw[gray,dotted] (0,0) grid (3,1);
 \draw[darkgray,->] (0,-.25)--(0,1.25); \draw[darkgray,->] (-.25,0)--(3.25,0);
 \fill[black] (0, 0) circle (.7mm) {};
 \fill[black] (1, 0) circle (.7mm) {};
 \fill[black] (2, 0) circle (.7mm) {};
 \fill[black] (3, 0) circle (.7mm) {};
 \fill[black] (0, 1) circle (.7mm) {};
  \fill[black] (1, 1) circle (.7mm) {};
  \draw[thick,bend right=15,postaction={decorate}] (0, 0) to node[below,blue]{1} (1, 0); 
  \draw[thick,bend right=15,postaction={decorate}] (1,0) to node[below,blue]{2,6} (2,0);
  \draw[thick,bend right=15,postaction={decorate}] (2,0) to node[above right,blue]{3,9} (1,1);
  \draw[thick,bend right=15,postaction={decorate}]  (1, 1) to node[left,blue]{10} (2, 0);
  \draw[thick,bend right=15,postaction={decorate}] (2,0) to node[below,blue]{7} (3, 0);
  \draw[thick,bend right=15,postaction={decorate}] (3,0) to node[above,blue]{8} (2,0); 
  \draw[thick,bend right=15,postaction={decorate}] (2,0) to node[above,blue]{11} (1,0);
  \draw[thick,bend right=15,postaction={decorate}] (1,0) to node[above,blue]{12} (0,0);
  \draw[thick,postaction={decorate}] (1,1) -- node[above,blue]{4} (0,1); 
  \draw[thick,postaction={decorate}] (0,1)-- node[above right,blue]{5} (1,0);
  \draw (0,0) node[below left]{$(0,0)$};
\end{tikzpicture}
\end{center}
\caption{An interval-closed set of the poset $\mathbf{A}_5$ with associated lattice paths $T$ and $B$, along with the associated lattice walk, with labels representing each step of the walk.}
\label{ex_typeA}
\end{figure}

We now obtain an expression for the generating function of $\IC(\bA_{n-1})$. 
\begin{thm}\label{thm:ICAn}
    The generating function of interval-closed sets of $\bA_{n-1}$ can be expressed as 
    $$\sum_{n\ge0} \card{\IC(\bA_{n-1})}z^{2n}=F(0,0,z),$$
    where $F(x,y):=F(x,y,z)$ satisfies the functional equation
\begin{equation}\label{eq:F}
F(x,y)= 1+z\left(x+\frac{1}{x}+\frac{x}{y}+\frac{y}{x}\right)F(x,y) - z \left(\frac{1}{x}+\frac{y}{x}\right)F(0,y)  - z\, \frac{x}{y} F(x,0)   - z^2\, \left(F(x,0)-F(0,0)\right).
\end{equation}
\end{thm}

\begin{proof}
By the bijection of Theorem~\ref{thm:walks_bijection}, we have $\card{\IC(\bA_{n-1})}=\card{\tWo_{2n}}$.
    In order to enumerate walks in $\tWo_{2n}$, we will consider more general walks that are not restricted to ending at the origin.
Let $\tWu_\ell$ be the set of walks in the first quadrant starting at the origin, and consisting of $\ell$ steps from the set $\{\e,\w,\se,\nw\}$ where no $\w$ step on the $x$-axis is immediately followed by an $\e$ step. Define the generating function
\begin{equation}\label{eq:Fdef} F(x,y,z)=\sum_{i,j,\ell\ge0}\card{\{W\in\tWu_\ell \text{ ending at }(i,j)\}}\,x^iy^jz^\ell, \end{equation}
and note that
$$F(0,0,z)=\sum_{\ell\ge0}\card{\{W\in\tWu_\ell \text{ ending at }(0,0)\}}\,z^\ell=\sum_{n\ge0}\card{\tWo_{2n}}z^{2n}=
\sum_{n\ge0} \card{\IC(\bA_{n-1})}z^{2n}.$$

The structure of the walks in $\tWu_\ell$ yields the following functional equation for $F(x,y):=F(x,y,z)$: 
\begin{align*} 
F(x,y)=& \ 1+z\left(x+\frac{1}{x}+\frac{x}{y}+\frac{y}{x}\right)F(x,y) &\\ 
& - z \left(\frac{1}{x}+\frac{y}{x}\right)F(0,y)  & \text{(no $\w$ or $\nw$ steps when on $y$-axis)}\\
& - z\, \frac{x}{y} F(x,0)  & \text{(no $\se$ steps when on $x$-axis)}\\  
& - z^2\, \left(F(x,0)-F(0,0)\right).  & \text{(no $\w$ followed by an $\e$ when on $x$-axis)}
\end{align*}
Indeed, a non-empty walk in $\tWu_n$ is obtained by appending a step to a walk in $\tWu_n$. This can be any step from among $\e,\w,\se,\nw$, with the following exceptions: 
\begin{itemize}
    \item we cannot append a $\w$ or $\nw$ step when on the $y$-axis --- the generating function for walks ending on the $y$-axis is $F(0,y)$,
    \item we cannot append a $\se$ step when on the $x$-axis --- the generating function for walks ending on the $x$-axis is $F(x,0)$,
    \item we cannot append an $\e$ step to a path ending with a $\w$ step on the $x$-axis ---the generating function for walks ending with a $\w$ step on the $x$-axis is $\dfrac{z}{x}\left(F(x,0)-F(0,0)\right)$, and the appended $\e$ step would contribute $xz$.\qedhere
\end{itemize}
\end{proof}

The functional equation~\eqref{eq:F} is similar to the one obtained when describing Gouyou-Beauchamps walks~\cite{GB}. These are like the walks in $\tWu_\ell$, but without the restriction that no $\w$ step on the $x$-axis is immediately followed by an $\e$ step. The functional equation for Gouyou-Beauchamps walks can be solved using the kernel method~\cite{MBM-MM}; alternatively, those ending at the origin are in bijection with pairs of nested Dyck paths, which can be easily counted by the Lindstr\"om--Gessel--Viennot lemma. However, we have not been able to solve the functional equation~\eqref{eq:F}. The term in the fourth line of the equation, which comes from our additional restriction, prevents some cancellations from occurring
when we try to apply the kernel method.

Still, equation~\eqref{eq:F} can be used to quickly generate hundreds of terms of the series expansion of $F(x,y,z)$ in the variable $z$.
In particular, the values of $\card{\IC(\bA_{n-1})}$ for $1\le n\le 10$ are $$1, 2, 8, 45, 307, 2385, 20362, 186812, 1814156, 18448851.$$

\subsection{Counting interval-closed sets of the type $B$ root poset}
\label{sec:typeB}

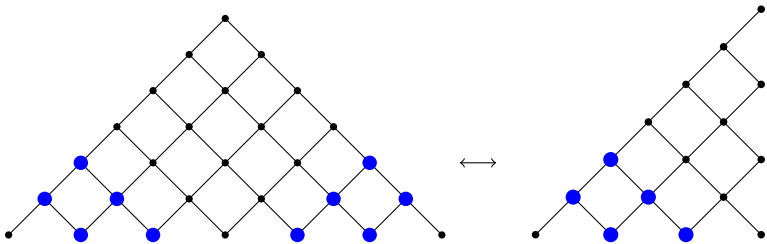
\begin{figure}[htbp]
\begin{center}
\begin{tikzpicture}[scale=.48]
\foreach \y in {0,...,6}
	{\foreach \x in {0,...,\y}
		{\fill (\x + \y, \y - \x) circle (0.1cm) {};
	         \ifthenelse{\x < \y}
			{\draw (\x + \y, \y - \x) -- (\x + \y + 1, \y - \x - 1);}{}
		\ifthenelse{\y < 6}
			{\draw (\x + \y, \y - \x) -- (\x + \y + 1, \y - \x+1);}{}
		}
	}

\fill[blue] (4, 0) circle (0.2cm) {};
\fill[blue] (3, 1) circle (0.2cm) {};
\fill[blue] (2, 0) circle (0.2cm) {};
\fill[blue] (1, 1) circle (0.2cm) {};
\fill[blue] (2, 2) circle (0.2cm) {};

\fill[blue] (9, 1) circle (0.2cm) {};
\fill[blue] (8,0) circle (0.2cm) {};
\fill[blue] (10, 0) circle (0.2cm) {};
\fill[blue] (11, 1) circle (0.2cm) {};
\fill[blue] (10, 2) circle (0.2cm) {};
\draw[<->] (12.5,2)--(13.5,2);
\end{tikzpicture}
\quad
\begin{tikzpicture}[scale=.5]
\draw (-2,-2) -- (4,4);
\draw (0,-2) -- (4,2);
\draw (2,-2) -- (4,0);
\draw (1,1) -- (4,-2);\draw (3,3) -- (4,2);
\draw (0,0) -- (2,-2);
\draw (-1,-1) -- (0,-2);
\draw (2,2) -- (4,0);
\fill (0,0) circle (0.1cm) {};
\fill (2,0) circle (0.1cm) {};
\fill (4,0) circle (0.1cm) {};
\fill (1,1) circle (0.1cm) {};
\fill (3,1) circle (0.1cm) {};
\fill (2,2) circle (0.1cm) {};
\fill (4,2) circle (0.1cm) {};
\fill (3,3) circle (0.1cm) {};
\fill (4,4) circle (0.1cm) {};
\fill (0,-2) circle (0.1cm) {};
\fill (2,-2) circle (0.1cm) {};
\fill (4,-2) circle (0.1cm) {};
\fill (1,-1) circle (0.1cm) {};
\fill (3,-1) circle (0.1cm) {};
\fill (-1,-1) circle (0.1cm) {};
\fill (-2,-2) circle (0.1cm) {};
\fill[blue] (2, -2) circle (0.2cm) {};
\fill[blue] (1, -1) circle (0.2cm) {};
\fill[blue] (0, -2) circle (0.2cm) {};
\fill[blue] (0, 0) circle (0.2cm) {};
\fill[blue] (-1, -1) circle (0.2cm) {};

\end{tikzpicture}
\end{center}
\caption{A symmetric interval-closed set of the poset $\mathbf{A}_7$, along with the corresponding interval-closed set of $\mathbf{B}_4$.}
\label{ex_typeB}
\end{figure}

One can study interval-closed sets of root posets of other Lie types. The next simplest is the type  $B_n$ root poset, which is the poset constructed as half of the poset $\bA_{2n-1}$, as seen in Figure~\ref{ex_typeB}. We denote this poset by $\bB_n$. See~\cite[Section 4.6]{BjornerBrenti} for background on root posets and \cite[Appendix]{RingelRootPosets} for diagrams of root posets of classical type. 

As illustrated in Figure~\ref{ex_typeB}, interval-closed sets of  $\bB_n$ are in bijection with vertically-symmetric interval-closed sets of $\bA_{2n-1}$. Note that in order for a symmetric ICS of $\bA_{m}$ to correspond to an ICS in a type $B$ root poset, $m$ must be odd. Below, we study symmetric ICS of $\bA_{m}$ for all $m$ (both odd and even) and then extract coefficients corresponding to ICS of type $B$ root posets.

Using the same construction from Section~\ref{sec:functional_equation}, symmetric order ideals of $\bA_{n-1}$ (that is, those invariant under vertical symmetry) are in bijection with symmetric Dyck paths in $\Dn$. Similarly, symmetric ICS of $\bA_{n-1}$ are in bijection with pairs of symmetric paths $(B,T)\in\DDn$ such that, in each maximal block where $B$ and $T$ coincide, up-steps come before down-steps. Such pairs of symmetric paths are uniquely determined by their left halves, which we denote by $(B_L,T_L)$. Each of $B_L$ and $T_L$ has $n$ steps from $\{\uu,\dd\}$, starts at the origin, and does not go below the $x$-axis. Because of the restrictions on $B$ and $T$, the path $B_L$ stays weakly below $T_L$, and in each maximal block where $B_L$ and $T_L$ coincide, up-steps come before down-steps. Additionally, $B_L$ and $T_L$ cannot end with a $\dd$ step where they coincide, since in the original paths $B$ and $T$, this $\dd$ would be followed by a $\uu$.

Finally, we apply the map from equation~\eqref{thm:walks_bijection} to the pair $(B_L,T_L)$, and note that $\dd$ steps where $B_L$ and $T_L$ coincide translate into $\w$ steps of the walk on the $x$-axis. This yields a bijection between symmetric ICS of $\bA_{n-1}$ and lattice walks in the first quadrant starting at the origin, consisting of $n$ steps from the set $\{ \e,\w,\se,\nw \}$ where no $\w$ step on the $x$-axis is immediately followed by an $\e$ step, and not ending with a $\w$ step on the $x$-axis. 
These are walks in the set $\tWu_n$, defined in the proof of Theorem~\ref{thm:ICAn}, with the additional restriction that they cannot end with a $\w$ step on the $x$-axis. The generating function for such restricted walks can be expressed in terms of the generating function $F(x,y,z)$ from Theorem~\ref{thm:ICAn}. 

\begin{thm}\label{thm:BrootGF}
    The generating function of symmetric interval-closed sets of $\bA_{n-1}$ can be expressed as 
    $$\sum_{n\ge0} \card{\IC_{\text{sym}}(\bA_{n-1})}z^{n}=F(1,1,z)-zF(1,0,z)+zF(0,0,z),$$
    where $F(x,y):=F(x,y,z)$ satisfies the functional equation~\eqref{eq:F}.
\end{thm}

\begin{proof}
From the generating function $F(x,y,z)$ defined in equation~\eqref{eq:Fdef}, we subtract the generating function for walks ending with a $\w$ step on the $x$-axis, which is $\frac{z}{x}\left(F(x,0,z)-F(0,0,z)\right)$. Setting $x=y=1$ to disregard the ending vertex, we obtain the desired expression.
\end{proof}

Using that $\IC_{\text{sym}}(\bA_{2n-1})=\IC(\bB_n)$, the coefficients of the even powers of the generating function in Theorem~\ref{thm:BrootGF} give the sequence $\card{\IC(\bB_n)}$, whose values for $1\le n\le 9$ are
$$ 
2, 
13, 
115, 
1166, 
12883, 
150912, 
1844322, 
23276741, 
301289155.$$

\subsection{Generalization to truncated rectangles}\label{ssec:truncated_rectangles}

The approach from Section~\ref{sec:functional_equation} can be generalized to count ICS of the poset $P_{m\times n;r}$ obtained by truncating the bottom $r$ ranks from $[m] \times [n]$  (see Figure~\ref{fig:truncated}). Throughout this section, we assume that $r\le m,n$. We will allow negative values of $r$, with the convention that $P_{m\times n;r}=P_{m\times n;0}$ if $r<0$. 
Note that $\bA_{n-1}=P_{n\times n;n}=P_{(n-1)\times (n-1);n-2}$. 
Using the bijection from Section~\ref{sec:rectangle}, order ideals of $P_{m\times n;r}$ are in bijection with lattice paths from $(0,n)$ to $(m+n,m)$ with steps $\uu)$ and $\dd$ that do not go below the line $y=r$. Denote this set of paths by $\Lmnr$, and let $\LLmnr=\{(B,T):B,T\in\Lmnr, B\le T\}$.

Interval-closed sets of of $P_{m\times n;r}$ are in bijection with equivalence classes of pairs $(B,T)\in\LLmnr$, where the equivalence relation allows us to change the portions of $B$ and $T$ where these two paths coincide. As before, we pick a representative of each equivalence class by requiring that, in each maximal block of steps where $B$ and $T$ coincide, up-steps come before down-steps. 

Let $\W^{h,s}_{\ell}$ be the set of lattice walks in the first quadrant starting at $(h,0)$ and ending at $(s,0)$, and consisting of $\ell$ steps from the set $\{\e,\w,\se,\nw\}$. 
The same map described in equation~\eqref{eq:wi} gives a bijection between pairs $(B,T)\in\LLmnr$ and walks $W\in\W^{n-r,m-r}_{m+n}$.
The condition that $B$ does not go below $y=r$ translates into the fact that $W$ stays in $x\ge0$ (since it never moves more than $n-r$ units to the left of where it started), and the condition that $B\le T$ translates into the fact that $W$ stays in $y\ge0$.

The above canonical representatives of the equivalence classes in $\LLmnr$ correspond to walks in $\W^{n-r,m-r}_{m+n}$ where no $\w$ step on the $x$-axis is immediately followed by an $\e$ step.
Let $\tW^{n-r,m-r}_{m+n}$ be the subset of walks with this property.

This discussion yields the following theorem.
\begin{thm}
\label{thm:walks_bijection_truncated}
    The set $\IC(P_{m\times n;r})$ of interval-closed sets of $P_{m\times n;r}$ is in bijection with the set $\tW^{n-r,m-r}_{m+n}$ of lattice walks in the first quadrant starting at $(n-r,0)$ and ending at $(m-r,0)$, and consisting of $m+n$ steps from the set $\{ \e,\w,\se,\nw \}$ where no $\w$ step on the $x$-axis is immediately followed by an $\e$ step.   
\end{thm}

\begin{example}
Figure~\ref{fig:truncated} shows an interval-closed set of $P_{4 \times 5; 1}$ with paths $T = \uu \dd \uu \dd \dd \uu \uu \dd \dd$ 
(in blue, dashed) and $B = \dd \dd \dd \dd \uu \uu \dd \uu \uu$ (in green). Its image is the quarter-plane walk
$W = \nw\,\w\,\nw\,\w\,\se\,\e\,\nw\,\se\,\se$ from $(4,0)$ to $(3,0)$.
\end{example}

\begin{figure}[htbp]
\begin{center}
\begin{tikzpicture}[scale=.5]
\foreach \y in {0,...,4}
	{\foreach \x in {0,...,3}
		{\ifthenelse{\x > -\y}{
			\fill[gray] (\x - \y, \y + \x) circle (0.1cm) {};
	          	\ifthenelse{\x < 3}
				{\draw[gray] (\x - \y, \y + \x) -- (\x - \y + 1, \y + \x + 1);}{}
			\ifthenelse{\y < 4}
				{\draw[gray] (\x - \y, \y + \x) -- (\x - \y - 1, \y + \x+1);}{}
						}{}
		}
	}
\fill[blue] (-1, 1) circle (0.2cm) {};
\fill[blue] (-2, 2) circle (0.2cm) {};
\fill[blue] (0, 2) circle (0.2cm) {};
\fill[blue] (2, 2) circle (0.2cm) {};
\fill[blue] (-3, 3) circle (0.2cm) {};
\fill[blue] (-1, 3) circle (0.2cm) {};
\fill[blue] (1, 3) circle (0.2cm) {};
\fill[blue] (3, 3) circle (0.2cm) {};
\fill[blue] (-4, 4) circle (0.2cm) {};
\fill[blue] (-2, 4) circle (0.2cm) {};
\fill[blue] (2, 4) circle (0.2cm) {};
\draw[ultra thick, babyblue, dashed] (-5,4)--(-4, 5)--(-3,4)--(-2,5)--(-1,4)--(0,3)--(1,4)--(2,5)--(3,4)--(4,3);
\draw[ultra thick, darkgreen] (-5,4)--(-4, 3)--(-3,2)--(-2,1)--(-1,0)--(0,1)--(1,2)--(2,1)--(3,2)--(4,3);
 \draw[darkgreen] (-3.75, 1.75) node {\large $B$};
  \draw[babyblue] (-3.75, 5.5) node { \large $T$};
    \draw[->] (5,3.5)--(6,3.5);
\end{tikzpicture}
\quad
\begin{tikzpicture}[scale=1.5,decoration={
    markings,
    mark=at position 0.55 with {\arrow{>}}}]
    \draw[gray,dotted] (0,0) grid (4,2);
 \draw[darkgray,->] (0,-.25)--(0,2.25); \draw[darkgray,->] (-.25,0)--(4.25,0);
 \fill[black] (4, 0) circle (0.7mm) {};
 \fill[black] (3, 0) circle (0.7mm) {};
 \fill[black] (3, 1) circle (0.7mm) {};
 \fill[black] (2, 1) circle (0.7mm) {};
 \fill[black] (1, 1) circle (0.7mm) {};
  \fill[black] (1, 2) circle (0.7mm) {};
  \fill[black] (0, 2) circle (0.7mm) {};
  \draw[thick,postaction={decorate}] (4, 0) to node[above,blue]{1} (3, 1); 
  \draw[thick,postaction={decorate}] (3,1) to node[above,blue]{2} (2,1);
  \draw[thick,bend right=15,postaction={decorate}] (2,1) to node[above right,blue]{3,7} (1,2);
  \draw[thick,postaction={decorate}]  (1, 2) to node[above,blue]{4} (0, 2);
  \draw[thick,postaction={decorate}] (0,2) to node[below left,blue]{5} (1,1);
  \draw[thick,postaction={decorate}] (1,1) to node[below,blue]{6} (2,1); 
  \draw[thick,bend right=15,postaction={decorate}] (1,2) to node[below left,blue]{8} (2,1);
  \draw[thick,postaction={decorate}] (2,1) to node[below left,blue]{9} (3,0);
  \draw (4,0) node[below]{$(4,0)$};
  \draw (3,0) node[below]{$(3,0)$};
\end{tikzpicture}
\end{center}
\caption{An interval-closed set of the poset $P_{4 \times 5; 1}$ with associated lattice paths $T$ and $B$, along with the associated lattice walk. }
\label{fig:truncated}
\end{figure}
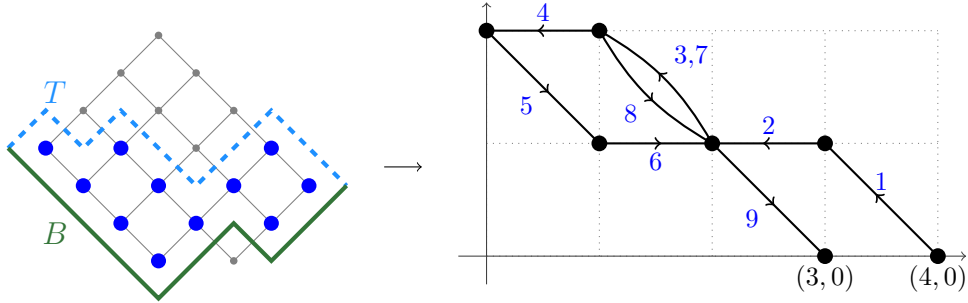

We use this bijection to prove the following functional equation for the generating function.
\begin{thm}\label{thm:ICP}
    The generating function of interval-closed sets of truncated rectangles can be expressed as 
    $$\sum_{m,n\ge \max(0, r)} \card{\IC(P_{m\times n;r})}t^{n-r}x^{m-r}z^{m+n}=G(t,x,0,z),$$
where $G(x,y):=G(t,x,y,z)$ satisfies the functional equation
\[
G(x,y)= \ \frac{1}{1-tx}+z\left(x+\frac{1}{x}+\frac{x}{y}+\frac{y}{x}\right)G(x,y) - z \left(\frac{1}{x}+\frac{y}{x}\right)G(0,y) 
- z\, \frac{x}{y} G(x,0) - z^2\, \left(G(x,0)-G(0,0)\right). 
\]
\end{thm}

\begin{proof}
    By the above bijection, we have $\card{\IC(P_{m\times n;r})}=\card{\tW^{n-r,m-r}_{m+n}}$.
Let $h=n-r$, $s=m-r$ and $\ell=m+n$. The inverse of this change of variables is $m=\frac{\ell+s-h}{2}$, $n=\frac{\ell-s+h}{2}$, $r=\frac{\ell-s-h}{2}$. Note that $\tW^{h,s}_{\ell}$ is empty unless $h+s+\ell\equiv0\bmod2$.
To enumerate walks in $\tW^{h,s}_{\ell}$, we will consider more general walks that are not restricted to ending at a particular point.
Let $\tW^{h}_{\ell}$ be the set of walks in the first quadrant, consisting of $\ell$ steps from among $\e,\w,\se,\nw$, starting $(h,0)$, and with no $\w$ step on the $x$-axis is immediately followed by an $\e$ step. Define the generating function
$$G_h(x,y,z)=\sum_{i,j,\ell\ge0}\card{\{W\in\tW^{h}_{\ell} \text{ ending at }(i,j)\}}\,x^iy^jz^\ell.$$
Now let
$$G(t,x,y,z)=\sum_{h\ge0} G_h(x,y,z)\,t^h,$$
whose evaluation at $y=0$ equals
$$G(t,x,0,z)=\sum_{h,s,\ell\ge0}\card{\tW^{h,s}_{\ell}}t^hx^sz^\ell=\sum_{\substack{h,s,\ell\ge0\\ h+s+\ell\equiv0\bmod2}}\card{\IC(P_{\frac{\ell+s-h}{2}\times\frac{\ell-s+h}{2};\frac{\ell-s-h}{2}})}t^hx^sz^{\ell}.$$
Note that, by definition, we have the symmetry $G(t,x,0,z)=G(x,t,0,z)$.

A similar argument to the one we used in the proof of Theorem~\ref{thm:ICAn}, by considering the possible last step of walks in $\tW^{h,s}_{\ell}$, shows that the generating function $G(x,y):=G(t,x,y,z)$ satisfies the stated functional equation.
\end{proof}

We have not been able to solve the functional equation in Theorem~\ref{thm:ICP}, but we can use it to compute hundreds of terms of the series expansion in the variable $z$. 
The first few terms of the series expansion of $G(t,x,0,z)$ are
\begin{multline*}\frac{1}{1-tx}\left(1+(t+x)\,z+(1+tx+t^2+x^2)\,z^2+
(2t+2x+2t^2x+2tx^2+t^3 + x^3)\,z^3\right.\\
\left.+(2+6tx+4t^2+4x^2+3t^3x+4tx^3+5t^2x^2+t^4+x^4)\,z^4+
\dots\right).\end{multline*}
The factor $\frac{1}{1-tx}$ appears because an ICS of $P_{m\times n;r}$ can also be viewed as an ICS of $P_{m\times n;r'}$ for any $r'<r$.
For example, for the poset $P_{3\times2;1}$, we have $h=2-1=1$, $s=3-1=2$, and $\ell=3+2=5$, so $\card{\IC(P_{3\times2;1})}$ is the coefficient of $tx^2z^5$ in $G(t,x,0,z)$, which is 24.

For the poset $P_{n\times n;n}$, we have $h=0$, $s=0$, and $\ell=2n$. In this case, $G_0(x,y,z)$ is the generating function $F(x,y,z)$ from Theorem~\ref{thm:BrootGF}, and the coefficient of $x^0$ in $G_0(x,0,z)$ is simply $F(0,0,z)$.

\subsection{Translating statistics between interval-closed sets and quarter-plane walks}
\label{sec:quarter_plane_stats}
Similarly to Theorem~\ref{thm:Motzkin_stats_bijection}, the quarter-plane walks obtained via the bijection from Theorem~\ref{thm:walks_bijection_truncated} give us information about the associated interval-closed set. 
In the next theorem, a  return of the walk to the $x$-axis is a $\se$ step ending on the $x$-axis, and a return to the $y$-axis is a $\w$ or $\nw$ step ending on the $y$-axis (we use this term even if the walk did not start on the $y$-axis).

\begin{thm} 
\label{statistics_walks}
    Let $I\in\IC(P_{m\times n;r})$, and let $W\in\tW^{n-r,m-r}_{m+n}$ be its image under the bijection from Theorem~\ref{thm:walks_bijection_truncated}. Then,
\begin{enumerate}[label=(\alph*)]
\item the cardinality of $I$ is the sum of the heights ($y$-coordinates) after each step of $W$,
\item the number of connected components of $I$ is the number of returns of $W$ to the $x$-axis, and
\item the number of minimal elements of $P_{m\times n;r}$ that are in $I$ is the number of returns of $W$ to the $y$-axis, not counting its last step.
\end{enumerate}
\end{thm}

\begin{proof}
Let $I\in\IC(P_{m\times n;r})$, and let $B$ and $T$ be the nested lattice paths from the bijection in Subsection~\ref{ssec:truncated_rectangles}. 
These are the same paths that we would have associated to $I$ in Subsection~\ref{ssec:bicolored} by viewing $I$ as an interval-closed set of $[m]\times[n]$.
Therefore, as in the proof of Theorem~\ref{thm:Motzkin_stats_bijection}, we have $|I|=\frac{1}{2} \sum_i d_i(B,T)$, where $d_i(B,T)$ is the distance between $B$ and $T$ after $i$ steps. 

Let us show that $\frac{1}{2}d_i(B,T)$ is also the $y$-coordinate at the end of the $i$-th step of the walk $W$. 
As in the proof of Theorem~\ref{thm:Motzkin_stats_bijection}, $\frac{1}{2}d_i(B,T)$ is the difference between the number of $\uu$ steps of $B$ and $T$ within their first $i$ steps. Using equation~\eqref{eq:wi}, this difference is equal to the number of $\nw$ steps (which occur in positions where $T$ has a $\uu$ step but $B$ does not) minus the number of $\se$ steps (which occur in positions where $B$ has a $\uu$ step but $T$ does not) within the first $i$ steps of $W$, which in turn equals the $y$-coordinate of $W$ after $i$ steps, noting that $W$ starts on the $x$-axis.
Summing over $i$, we obtain part~(a).

The connected components of $I$ correspond to the maximal blocks where $B$ is strictly below $T$, or equivalently, $W$ is strictly above the $x$-axis. Part~(b) follows.

To prove part~(c), observe that, after any given number of steps, the height of the path $B$ always equals the $x$-coordinate of the walk $W$. This is because $W$ starts at $(n-r,0)$ and, by equation~\eqref{eq:wi}, 
its $i$th step $w_i$ is $\se$ or $\e$ if $b_i=\uu$, and it is $\nw$ or $\w$ if $b_i=\dd$.
An element of $P_{m\times n;r}$ is in $I$ if and only if it is above $B$ and below $T$. With the exception of its last step, the path $B$ goes below a minimal element when it reaches height $0$, which happens precisely when $W$ returns to the $y$-axis.
By construction, the path $T$ does not lie below any minimal element of the poset. Hence, the minimal elements of $P_{m\times n;r}$ that are in $I$ correspond exactly to the returns of $W$ to the $y$-axis, not counting the last step of~$W$.
\end{proof}

\begin{example}
    Let $I$ be the interval-closed set of the poset $\mathbf{A}_5$ drawn in Figure \ref{ex_typeA}. Then $|I|=3$, which equals the sum of the heights of the corresponding walk $W$ after steps 3, 4 and 9, all at height $1$, since all the other steps end at height $0$. Also, $I$ has two connected components, corresponding to the two returns of $W$ to the $x$-axis, after steps 5 and 10. Finally, $I$ contains one of the minimal elements of the poset, corresponding to the return of $W$ to the $y$-axis after step 4 (the other return, after step 12, is not counted since it is the last step).
\end{example}
\begin{example} 
Let $I$ be the interval-closed set of $P_{4\times 5;1}$ drawn in Figure \ref{fig:truncated}. Then $|I|=11$, which equals the sum of the heights of the steps 
of the corresponding walk $W$ (three steps ending at height~$2$, and five steps ending at height~$1$). Here $I$ has only one connected component, and $W$ returns to the $x$-axis only once (at the end). Finally, $I$ contains one minimal element of the poset, which corresponds to the return of $W$ to the $y$-axis after step $4$.
\end{example}

\section{Future directions}
\label{sec:future}
It would be interesting to enumerate interval-closed sets of other posets. One natural generalization of the poset $[m]\times[n]$ is the product of three chains $[\ell]\times[m]\times[n]$. Table~\ref{tab:ICS_2_by_m_by_n} lists the number of interval-closed sets of $[2]\times[m]\times[n]$ for small values of $m$ and $n$, computed in SageMath~\cite{sage}. Other posets of interest include root and minuscule posets of other Lie types.
\begin{table}[htbp]
    \centering
    \begin{tabular}{r|c|c|c|c|c|c|c}
         $m$\ \textbackslash\ $n$ & 2 &3&4&5& 6 & 7 & 8\\ 
         \hline
         2 & 101 &526 & 2,085& 6,793 & 19,100 & 47,883 & 109,501\\
         3& 526& 5,030 & 33,792 &175,507 &749,468 & 2,743,751 &8,870,441\\
         4 &2,085 & 33,792& 361,731 & 2,851,562 & 17,768,141 & 91,871,593 &408,168,856\\
         5 &6,793&175,507&2,851,562&32,797,595& 288,594,237 &2,050,193,127 & 12,225,400,806
    \end{tabular}
    \caption{The number of interval-closed sets of $[2]\times [m]\times [n]$ for small values of $m$ and $n$.}
    \label{tab:ICS_2_by_m_by_n}
\end{table}

\section*{Acknowledgements}
The authors thank Torin Greenwood for helpful conversations and the developers of SageMath~\cite{sage}, which was useful in this research. This work benefited from opportunities to present and collaborate at the Fall 2023 AMS Central Sectional Meeting as well as the June 2024 conference on Statistical and Dynamical Combinatorics at MIT.
Lewis was partially supported by Simons Collaboration Grant \#634530 and gift MPS-TSM-00006960.
Elizalde was partially supported by Simons Collaboration Grant \#929653.
Striker was supported by a Simons Foundation gift MP-TSM-00002802 and NSF grant DMS-2247089. 

\bibliographystyle{abbrv}
\bibliography{master}

\end{document}